\documentclass[aos,preprint,reqno,10pt]{imsart}

\usepackage[a4paper, total={6.5in, 9in}]{geometry}
\setattribute{journal}{name}{}

\usepackage{amsfonts,amssymb,amsthm,amsmath,bbm}
\RequirePackage[colorlinks,citecolor=blue,urlcolor=blue]{hyperref}
\usepackage{graphicx}
\usepackage{subfigure}

\allowdisplaybreaks

\startlocaldefs
\numberwithin{equation}{section}
\theoremstyle{plain}
\newtheorem{theorem}{Theorem}[section]
\newtheorem{lemma}{Lemma}[section]
\newtheorem{corollary}{Corollary}[section]
\newtheorem{proposition}{Proposition}[section]

\theoremstyle{remark}
\newtheorem*{remark}{Remark}

\newcommand{\indI}{\mathbbm{1}}
\def\ci{\perp\!\!\!\perp}

\endlocaldefs

\begin{document}

\begin{frontmatter}

% "Title of the paper"
\title{Fractional Poisson Fields and Martingales}
\runtitle{Fractional Poisson Fields and Martingales}

% indicate corresponding author with \corref{}
% \author{\fnms{John} \snm{Smith}\corref{}\ead[label=e1]{smith@foo.com}\thanksref{t1}}
% \thankstext{t1}{Thanks to somebody} 
% \address{line 1\\ line 2\\ printead{e1}}
% \affiliation{Some University}

\author{\fnms{Giacomo} \snm{Aletti}\ead[label=e1]{giacomo.aletti@unimi.it}\thanksref{t1}}
\address{\printead{e1}}
\author{\fnms{Nikolai} \snm{Leonenko}\ead[label=e2]{LeonenkoN@Cardiff.ac.uk}\thanksref{t2}}
\address{\printead{e2}}
\author{\fnms{Ely} \snm{Merzbach}\ead[label=e3]{ely.merzbach@biu.ac.il}\thanksref{t3}}
\address{\printead{e3}}
%\author{\fnms{???} \snm{???}\ead[label=e2]{???}}
%\address{\printead{e2}}
\affiliation{\thanksmark{t1}ADAMSS Center \& Universit\`a degli Studi di Milano, Italy\\
\thanksmark{t2}Cardiff University, United Kingdom\\
\thanksmark{t3}Bar-Ilan University, Israel}

\thankstext{t2}{N. Leonenko was supported in particular by Cardiff Incoming Visiting
Fellowship Scheme and International Collaboration Seedcorn Fund and
Australian Research Council's Discovery Projects funding scheme (project
number DP160101366)}

\runauthor{G. Aletti, N. Leonenko, E. Merzbach}

\begin{abstract}
We present new properties for the Fractional Poisson process and the Fractional Poisson field on the plane. A martingale characterization for Fractional Poisson processes is given. We extend this result to Fractional Poisson fields, obtaining some other characterizations. The fractional differential equations are studied. We consider a more general Mixed-Fractional Poisson process and show that this process is the stochastic solution of a system of fractional differential-difference equations. Finally, we give some simulations of the Fractional Poisson field on the plane.
\end{abstract}



\end{frontmatter}

There are several different approaches to the fundamental concept of Fractional Poisson process (FPP) on the real line. The ``renewal'' definition extends the characterization of the Poisson process as a sum of independent non-negative exponential random variables. If one changes the law of interarrival times to the Mittag-Leffler distribution (see \cite{MGS,MGV1,RS}), the FPP arises. 
A second approach is given in \cite{BO1}, where the renewal approach to the Fractional Poisson process is developed and it is proved that its one-dimensional distributions coincide with the solution to fractionalized state probabilities. In \cite{MNV} it is shown that a kind of Fractional Poisson process can be constructed by using an ``inverse subordinator'', which leads to a further approach. 

In \cite{LM}, following this last method, the FPP is generalized and defined afresh, obtaining a Fractional Poisson random field (FPRF) parametrized by points of the Euclidean space $\mathbb{R}_+^{2} $, in the same spirit it has been done before for Fractional Brownian fields, see, e.g., \cite{HM1,IM2,IM1,LRMT}.

The starting point of our extension will be the set-indexed Poisson process
which is a well-known concept, see, e.g., \cite{HM1,IM1,MN,MS,SKM}.

In this paper, we first present a martingale characterization of the Fractional
Poisson process. We extend this
characterization to FPRF
using the concept of increasing path and strong martingales. 
This characterization permits us to
give a definition of a set-indexed Fractional Poisson process. We study
the fractional differential equation for FPRF.
Finally, we study
Mixed-Fractional Poisson processes.

The paper is organized as follows. In the next section, we collect
some known results from the theory of subordinators
and inverse subordinators, see \cite{B,MS1,VTM1,VTM2} among others. 
In Section~\ref{sect:3}, we prove a martingale characterization of the FPP, which is a generalization of the Watanabe Theorem.  In Section~\ref{sect:6}, another generalization called ``Mixed-Fractional Poisson process'' is introduced and some distributional properties are studied as well as Watanabe characterization
is given. Section~\ref{sect:4} is devoted to FPRF. We begin by computing covariance for this process, then we give some characterizations using increasing paths and intensities. We present a Gergely-Yeshow characterization and discuss random time changes. Fractional differential equations are discussed on Section~\ref{sect:5}. 

Finally, we present some simulations for the FPRF.  

%Other different generalizations of FPP and fields can be found in \cite{WW,WWZ,BDE,MR2,MR}.

\section{Inverse Subordinators}

This section collects some known resuts from the theory of subordinators and inverse subordinators \cite{B,MS1,VTM1,VTM2}.

\subsection{Subordinators and their inverse}

Consider an increasing L\'{e}vy process $L=\{L(t),\ t\geq 0\},$ starting
from $0,$ which is continuous from the right with left limits (cadlag),
continuous in probability, with independent and stationary increments. Such
a process is known as a L\'{e}vy subordinator with Laplace exponent 
\begin{equation*}
\phi (s)=\mu s+\int_{(0,\infty )}(1-e^{-sx})\Pi (dx),\quad s\geq 0,
\label{LS}
\end{equation*}
where $\mu \geq 0$ is the drift and the L\'{e}vy measure $\Pi $ on $\mathbb{R%
}_{+}\cup \left\{ 0\right\} $ satisfies 
\begin{equation*}
\int_{0}^{\infty }\min (1,x)\Pi (dx)<\infty .  %\label{LS1}
\end{equation*}

This means that 
\begin{equation*}
\mathrm{E}e^{-sL(t)}=e^{-t\phi (s)},\ s\geq 0. % \label{LS2}
\end{equation*}

Consider the inverse subordinator $Y(t),\ t\geq 0,$ which is given by the
first-passage time of~$L:$%
\begin{equation*}
Y(t)=\inf \left\{ u\geq 0:L(u)>t\right\} ,t\geq 0.  %\label{LS3}
\end{equation*}%
The process $Y(t),\ t\geq 0,$ is non-decreasing and its sample paths are
a.s.\ continuous if $L$ is
strictly increasing. 
%Also $Y$ is, in general, non-Markovian with
%non-stationary and non-independent increments.

We have 
\begin{equation*}
\left\{ (u_{i},t_{i})\colon L(u_{i})<t_{i},i=1,\ldots ,n\right\} =\left\{ (u_{i},t_{i})\colon 
Y(t_{i})>u_{i},i=1,\ldots ,n\right\} ,  %\label{IS1}
\end{equation*}
%and for any $z>0$ 
%\begin{equation*}
%\mathrm{P}\left\{ Y(t)>x\right\} = \mathrm{P}\left\{ L(x)\leq t\right\} =%
%\mathrm{P}\left\{ e^{-zL(x)}\geq e^{-zt}\right\} \leq \exp \{tz-x\phi (z)\},
%\end{equation*}
%which implies that for any $p>0,\mathrm{E}Y^{p}(t)<\infty $, since
%\begin{align*}
%\mathrm{E}Y^{p}(t) = 
%% \int_0^\infty x^p \, d\mathrm{P}\{ Y(t) \leq x\} =
%p \int_0^\infty x^{p-1} (1 -\mathrm{P}\{ Y(t) \leq x\}) dx \leq 
%p e^{tz} \int_0^\infty x^{p-1} e^{-x\phi (z)} dx < \infty.
%\end{align*}
and it is known \cite{M-Ta,PSW05,VTM1,VTM2} that for any $p>0,\mathrm{E}Y^{p}(t)<\infty $.

Let $U(t)=\mathrm{E}Y(t)$ be the renewal function. Since
\begin{equation*}
\tilde{U}(s)=\int_{0}^{\infty }U(t)e^{-st}dt=%
\frac{1}{s\phi (s)},  %\label{IS5}
\end{equation*}%
then $\tilde{U}$ characterizes the inverse process $Y$, since $\phi $
characterizes $L.$

We get a covariance formula \cite{VTM1,VTM2}
\begin{equation*}
\mathrm{Cov}(Y(t),Y(s))=\int_{0}^{\min
(t,s)} \mspace{-50mu} (U(t-\tau )+U(s-\tau ))dU(\tau )-U(t)U(s).  %\label{IS6}
\end{equation*}

The most important example is considered in the next section, but there are
some other examples.

\subsection{Inverse stable subordinators}

Let $L_{\alpha }=\{L_{\alpha }(t),t\geq 0\}$, be an $\alpha -$stable
subordinator with $\phi (s)=s^{\alpha },0<\alpha <1$. %and
The density of $L_{\alpha }(1)$ is of the form \cite{UZ_SD99}
% whose density $g(t,x)$ is such that $L_{\alpha }(1)$ has pdf on $x>0$ given by 
\begin{equation}
g_{\alpha }(x)%=g(1,x)
=\frac{1}{\pi }\mathop{\displaystyle \sum}%
\limits_{k=1}^{\infty }(-1)^{k+1}\frac{\Gamma (\alpha k+1)}{k!}\frac{1}{%
x^{\alpha k+1}}\sin (\pi k\alpha )=\frac{1}{x}W_{-\alpha ,0}(-x^{-\alpha
}).  \label{10}
\end{equation}%

Here we use the Wright' s generalized Bessel function (see, e.g., \cite{HMS}%
) 
\begin{equation}
W_{\gamma ,\beta }(z)=\mathop{\displaystyle \sum}\limits_{k=0}^{\infty }%
\frac{z^{k}}{\Gamma (1+k)\Gamma (\beta +\gamma k)},\quad z\in \mathbb{C},
\label{WR}
\end{equation}%
where $\gamma >-1,$ and $\beta \in \mathbb{R}$.
The set of jump times of $L_{\alpha }$ is a.s. dense. The L\'{e}vy
subordinator is strictly increasing, since the process $L_{\alpha}$ admits a density.

Then the inverse stable subordinator

\begin{equation*}
Y_{\alpha }(t)=\inf \{u\geq 0:L_{\alpha }(u)>t\}
\end{equation*}
has density \cite[p.110]{MS1} (see also \cite{PolSca16})
\begin{equation}
f_{\alpha }(t,x)=\frac{d}{dx}\mathrm{P}\{Y_{\alpha }(t) \leq x\}
=\frac{t}{\alpha }x^{-1-\frac{1}{\alpha }}g_{\alpha }(tx^{-%
\frac{1}{\alpha }}),\quad x> 0,\quad t>0.  \label{D1}
\end{equation}
The Laplace transform of the density $f_{a}(t,x)$ is
\begin{equation}
\int_{0}^{\infty }e^{-st}f_{\alpha}(t,x)dt=s^{\alpha -1}e^{-xs^{\alpha }},\quad
s\geq 0,  \label{L1}
\end{equation}

Its paths are continuous and nondecreasing. For $\alpha =1/2,$ the inverse
stable subordinator is the running supremum process of Brownian motion, and
for $\alpha \in (0,1/2)$ this process is the local time at zero of a
strictly stable L\'{e}vy process of index $\alpha /(1-\alpha ).$

Let 
\begin{equation}
E_{\alpha }(z)=\sum_{k=0}^{\infty }\frac{z^{k}}{\Gamma (\alpha k+1)},\
\alpha >0,\ z\in \mathbb{C}  \label{ML1}
\end{equation}%
be the Mittag-Leffler function \cite{HMS}, and recall the following:

{i) The Laplace transform of function $E_{\alpha }(-\lambda t^{\alpha })$ is of the form 
\begin{equation*}
\int_{0}^{\infty }e^{-st}E_{\alpha }(-\lambda t^{\alpha })dt=\frac{s^{\alpha -1}}{%
\lambda +s^{\alpha }},\quad 0<\alpha <1,\ t\geq 0, \Re (s)>|\lambda|^{1/\alpha}.
\end{equation*}%
}

(ii) The function $E_{\alpha }(\lambda t^{\alpha })$ is an eigenfunction at the 
the fractional Caputo-Djrbashian derivative $\mathrm{D}_{t}^{\alpha }$ with eigenvalue
$\lambda$ \cite[p.36]{MS1}
\begin{equation*}
\mathrm{D}_{t}^{\alpha }E_{\alpha }(\lambda t^{\alpha })=\lambda E_{\alpha }(\lambda t^{\alpha
}),\quad 0<\alpha <1, \lambda \in \mathbb{R},
\end{equation*}%
where $\mathrm{D}_{t}^{\alpha }$
is defined as (see \cite{MS1}) 
\begin{equation}
\mathrm{D}_{t}^{\alpha }u(t)=\frac{1}{\Gamma (1-\alpha )}\int_{0}^{t}\frac{%
du(\tau )}{d\tau }\frac{d\tau }{(t-\tau )^{\alpha }},\quad 0<\alpha <1.
\label{FD}
\end{equation}
Note that the classes of functions for which the 
Caputo-Djrbashian derivative is well defined are discussed in 
\cite[Sections 2.2. and 2.3]{MS1}
(in particular one can use the class of absolutely continuous functions).

\begin{proposition}\label{prop:2.1} 
The $\alpha$-stable inverse subordinators satisfy the following properties:
\begin{enumerate}
\item[(i)] 
\begin{equation*}
\mathrm{E}e^{-sY_{\alpha }(t)}=\sum_{n=0}^{\infty }\frac{(-st^{\alpha })^{n}%
}{\Gamma (\alpha n+1)}=E_{\alpha }(-st^{\alpha }),\quad s>0.
\end{equation*}

\item[(ii)] Both processes $L_{\alpha }(t),t\geq 0$ and $Y_{\alpha }(t)$ are
self-similar 
\begin{equation*}
\frac{L_{\alpha }(at)}{a^{1/\alpha }}\overset{d}{=}L_{\alpha }(t),\;\frac{%
Y_{\alpha }(at)}{a^{\alpha }}\overset{d}{=}Y_{\alpha }(t),\quad a>0.
\end{equation*}

\item[(iii)] For $0<t_{1}<\cdots <t_{k}$,
\begin{equation*}
\frac{\partial^{k} \mathrm{E}(Y_{\alpha }(t_{1})\cdots Y_{\alpha }(t_{k}))}{%
\partial t_{1}\cdots \partial t_{k}}=\frac{1}{\Gamma ^{k}(\alpha )}\frac{1}{%
\left[ t_{1}(t_{2}-t_{1})\cdots (t_{k}-t_{k-1})\right] ^{1-\alpha }}.
\end{equation*}

In particular,

\noindent \textrm{(A)} 
\begin{equation*}
\mathrm{E}Y_{\alpha }(t)=\frac{t^{\alpha }}{\Gamma (1+\alpha )};\mathrm{E[}%
Y_{\alpha }(t)]^{\nu }=\frac{\Gamma (\nu +1)}{\Gamma (\alpha \nu +1)}%
t^{\alpha \nu },\quad\nu >0;
\end{equation*}%

\noindent \textrm{(B)} 
\begin{equation}
%\begin{multline}
\mathrm{Cov}(Y_{\alpha }(t),Y_{\alpha }(s))
%\\
=\frac{1}{\Gamma (1+\alpha )\Gamma (\alpha )}\int_{0}^{\min (t,s)}\left(
(t-\tau )^{\alpha }+(s-\tau )^{\alpha }\right) \tau ^{\alpha -1}d\tau -\frac{%
(st)^{\alpha }}{\Gamma ^{2}(1+\alpha )}.  \label{COV4}
%\end{multline}%
\end{equation}%

\end{enumerate}
\end{proposition}
\begin{proof}
See \cite{B,VTM1,VTM2}.
\end{proof}

%\begin{remark}
%There is a (complicated) form of all finite-dimensional
%distributions of $Y_{\alpha }(t),$ $t\geq 0,$ in the form of Laplace
%transforms, see \cite{B}.
%\end{remark}

\subsection{Mixture of inverse subordinators}

This subsection collects some results from the theory of inverse subordinators,
see \cite{VTM1,VTM2,MS1,BEG,LMS}.

Different kinds of inverse subordinators can be considered.

Let $L_{\alpha _{1}}$ and $L_{\alpha _{2}}$ be two independent stable
subordinators.
The mixture of them $L_{\alpha _{1},\alpha
_{2}}=\{L_{\alpha _{1},\alpha _{2}}(t),t\geq 0\}$ is defined by its Laplace
transform: for $ s\geq 0,\ C_{1}+C_{2}=1,\ C_{1}\geq 0,\
C_{2}\geq 0,\ \alpha _{1}<\alpha _{2}$,
\begin{equation}
\mathrm{E}e^{-sL_{\alpha _{1},\alpha _{2}}(t)}=\exp \{-t(C_{1}s^{\alpha
_{1}}+C_{2}s^{\alpha _{2}})\}.  \label{M1}
\end{equation}
It is possible to prove that
\begin{equation*}
L_{\alpha _{1},\alpha _{2}}(t)=(C_{1})^{\frac{1}{\alpha _{1}}}L_{\alpha
_{1}}(t)+(C_{2})^{\frac{1}{\alpha _{2}}}L_{\alpha _{2}}(t),\quad t\geq 0,
\end{equation*}
is not self-similar, unless $\alpha _{1}=\alpha _{2}=\alpha ,$ since $L_{\alpha
_{1},\alpha _{2}}(at)=^{d}(C_{1})^{\frac{1}{\alpha _{1}}}a^{\frac{1}{\alpha
_{1}}}L_{\alpha _{1}}(t)+(C_{2})^{\frac{1}{\alpha _{2}}}a^{\frac{1}{\alpha
_{2}}}L_{\alpha _{2}}(t).$ This expression is equal to $a^{\frac{1}{\alpha }%
}L_{\alpha _{1},\alpha _{2}}(t)$  for any $t>0$ if and only if
$\alpha _{1}=\alpha
_{2}=\alpha ,$ in which case the process $L_{\alpha _{1},\alpha _{2}}$ can
be reduced to the classical stable subordinator (up to a constant).

The inverse subordinator is defined by

\begin{equation}
Y_{\alpha _{1},\alpha _{2}}(t)=\inf \{u\geq 0:L_{\alpha _{1},\alpha
_{2}}(u)>t\},\quad t\geq 0.  \label{ISM}
\end{equation}
We assume that $C_{2}\neq 0$ without loss of generality (the case $C_{2}=0$
reduces to the previous case of single inverse subordinator).

It was proved in \cite{LMS} that 
\begin{equation}
\tilde{U}(t)=\frac{1}{(C_{1}s^{\alpha _{1}}+C_{2}s^{\alpha _{2}})s},U(t)=%
\frac{1}{C_{2}}t^{\alpha _{2}}E_{\alpha _{2}-\alpha _{1},\alpha
_{2}+1}(-\tfrac{C_{1}}{C_2}t^{_{\alpha _{2}-\alpha _{1}}}),  \label{UI}
\end{equation}%
where $E_{\alpha ,\beta }(z)$ is the two-parametric Generalized
Mittag-Leffler function (\cite{DAL,HMS}) 
\begin{equation*}
E_{\alpha ,\beta }(z)=\sum_{k=0}^{\infty }\frac{z^{k}}{\Gamma (\alpha
k+\beta )},\ \alpha >0,\beta >0,\ z\in \mathbb{C}.  %\label{ML21}
\end{equation*}%
Also for the Laplace transform of the density $f_{\alpha _{1},\alpha
_{2}}(t,u)=\frac{d}{du}\mathrm{P}\{Y_{\alpha _{1},\alpha _{2}}(t)\leq u\},$ $%
u\geq 0,$ of the inverse subordinator $Y_{\alpha _{1},\alpha _{2}}=\{Y_{\alpha
_{1},\alpha _{2}}(t),t\geq 0\},$ we have the following expression \cite{MeS}: 
\begin{equation}
\tilde{f}_{\alpha _{1},\alpha _{2}}(s,u)=\int_{0}^{\infty }e^{-st}f_{\alpha
_{1},\alpha _{2}}(t,u)dt=\frac{1}{s}[C_{1}s^{\alpha _{1}}+C_{2}s^{\alpha
_{2}}]e^{-u[C_{1}s^{\alpha _{1}}+C_{2}s^{\alpha _{2}}]},\quad s\geq 0,
\label{M2}
\end{equation}%
and the Laplace transform of $\tilde{f}$ is given by 
\begin{equation}
\int_{0}^{\infty }e^{-pu}\tilde{f}_{\alpha _{1},\alpha _{2}}(s,u)du=\frac{%
\phi (s)}{s(p+\phi (s))}=\frac{C_{1}s^{\alpha _{1}-1}+C_{2}s^{\alpha _{2}-1}%
}{p+C_{1}s^{\alpha _{1}}+C_{2}s^{\alpha _{2}}},\quad p\geq 0.  \label{M3}
\end{equation}

From \cite[Theorem 2.3]{BEG} we have the following expression for $u\geq 0,t> 0$:
\begin{equation*}
f_{\alpha _{1},\alpha _{2}}(t,u)=\frac{C_{1}}{\lambda t^{\alpha _{1}}}%
\sum_{r=0}^{\infty }\frac{1}{r!}(-\frac{C_{2}\left\vert u\right\vert }{%
\lambda t^{\alpha _{2}}})^{r}W_{-\alpha _{1},1-\alpha _{2}r-\alpha _{1}}(-%
\frac{C_{1}\left\vert u\right\vert }{\lambda t^{\alpha _{1}}})+
\end{equation*}%
\begin{equation}
+\frac{C_{2}}{\lambda t^{\alpha _{2}}}\sum_{r=0}^{\infty }\frac{1}{r!}(-%
\frac{C_{1}\left\vert u\right\vert }{\lambda t^{\alpha _{1}}})^{r}W_{-\alpha
_{2},1-\alpha _{1}r-\alpha _{2}}(-\frac{C_{2}\left\vert u\right\vert }{%
\lambda t^{\alpha _{2}}}).  \label{MD2}
\end{equation}%
%where $W_{\gamma ,\beta }(z)$ is the Wright's generalized Bessel function (%
%\ref{WR}).

One can also consider the tempered stable inverse subordinator, the inverse
subordinator to the Poisson process, the compound Poisson process with positive jumps, the Gamma
and the inverse Gaussian L\'{e}vy processes. For additional details see
%Veillette and Taqqu 
\cite{LMS,VTM1,VTM2}.

\section{Fractional Poisson Processes and Martingales}\label{sect:3}

\subsection{Preliminaries}

\bigskip The first definition of FPP $N_{\alpha }=\{N_{\alpha }(t),t\geq 0\}$
was given in %by Mainardi et al. 
\cite{MGS} (see also \cite{MGV1}) as a renewal
process with Mittag-Leffler waiting times between the events 
\begin{equation*}
N_{\alpha }(t)=\max \left\{ n:T_{1}+...+T_{n}\leq t\right\}=\sum_{j=1}^{\infty }{\indI}_{\{ T_{1}+...+T_{j}\leq t \}}
,\qquad
t\geq 0,  %\label{DF1}
\end{equation*}
where $\left\{ T_{j}\right\} ,$ $j=1,2,\ldots$ are iid random variables with
the strictly monotone Mittag-Leffler distribution function 
\begin{equation*}
F_{\alpha }(t)=\mathrm{P}\left( T_{j}\leq t\right) =1-E_{\alpha }(-\lambda
t^{\alpha }),\qquad t\geq 0,0<\alpha <1,\quad j=1,2,\ldots
\end{equation*}
% Meerschaert et al. 
The following stochastic representation
for FPP is found in \cite{MNV}:
\begin{equation*}
N_{\alpha }(t)=N(Y_{\alpha }(t)),\quad t\geq 0,\quad \alpha \in (0,1),
\end{equation*}
where $N=\{N(t),t\geq 0\},$ is the classical homogeneous Poisson process
with parameter $\lambda >0,$ which is independent of the inverse stable
subordinator $Y_{\alpha }.$
One can compute the following expression for the one-dimensional
distribution of FPP (see \cite{ScaGroMai04}): 
\begin{align*}
\mathrm{P}\left( N_{\alpha }(t)=k\right)& =p_{k}^{(\alpha
)}(t)=\int_{0}^{\infty }\frac{e^{-\lambda x}(\lambda x)^{k}}{k!}f_{\alpha
}(t,x)dx \\
%&=\frac{t\lambda ^{k}}{\alpha k!}\int_{0}^{\infty }e^{-\lambda
%x}x^{k-1-\frac{1}{\alpha }}g_{\alpha }(tx^{-\frac{1}{\alpha }})dx \\
& =\frac{(\lambda t^{\alpha })^{k}}{k!}\mathop{\displaystyle \sum}%
\limits_{j=1}^{\infty }\frac{(k+j)!}{j!}\frac{(-\lambda t^{\alpha })^{j}}{%
\Gamma (\alpha (j+k)+1)}=\frac{\left( \lambda t^{\alpha }\right) ^{k}}{k!}%
E_{\alpha }^{(k)}(-\lambda t^{\alpha }) \\
&=(\lambda t^{\alpha })^{k}E_{\alpha ,\alpha k+1}^{k+1}(-\lambda t^{\alpha
}),\quad k=0,1,2...,t\geq 0,\quad 0<\alpha <1,
\end{align*}%
where $f_{\alpha}$ is given by (\ref{D1}), 
$E_{\alpha }(z)$ is the Mittag-Leffler function 
(\ref{ML1}),
%evaluated at $z=-\lambda t^{\alpha }$%
$E_{\alpha }^{(k)}(z)$ is the $k-$th
derivative of $E_{\alpha }(z)$,
% evaluated at $z=-\lambda t^{\alpha },$ 
and $E_{\alpha ,\beta }^{\gamma }(z)$ is the three-parametric Generalized Mittag-Leffler function defined as
follows \cite{HMS,P}: 
\begin{equation}
E_{\alpha ,\beta }^{\gamma }(z)=\sum_{j=0}^{\infty }\frac{(\gamma )_{j}z^{j}%
}{j!\Gamma (\alpha j+\beta )},\alpha >0,\beta >0,\ \gamma >0,\quad z\in 
\mathbb{C},  \label{ML3}
\end{equation}
where 
\[
(\gamma )_{j} = 
\begin{cases}
1 & \text{if } j = 0;
\\
\gamma (\gamma+1) \cdots (\gamma+j - 1)
& \text{if } j = 1, 2, \ldots
\end{cases}
\]
is the Pochhammer symbol.

Finally, %Beghin and Orsingher 
in \cite{BO1,BO2} it is shown that the
marginal distribution of FPP satisfies the following system of fractional
differential-difference equations (see \cite{Las03}):%
\begin{equation*}
\mathrm{D}_{t}^{\alpha }p_{k}^{(\alpha )}(t)=-\lambda (p_{k}^{(\alpha
)}(t)-p_{k-1}^{(\alpha )}(t)),\quad k=0,1,2, \ldots
\end{equation*}%
with initial conditions: $p_{0}^{(\alpha )}(0)=1,p_{k}^{(\alpha
)}(0)=0,k\geq 1,$ and $p_{-1}^{(\alpha )}(t)=0,$ where $\mathrm{D}%
_{t}^{\alpha }$ is the fractional Caputo-Djrbashian derivative (\ref{FD}).
See also \cite{B-T2}.

\begin{remark}
Note that 
\begin{equation*}
\mathrm{E}N_{\alpha }(t)=
\mathrm{E}\big[\mathrm{E}[N(Y_{\alpha }(t)) | Y_{\alpha }(t)]\big]=
\int_{0}^{\infty }[\mathrm{E}N(u)]f_{\alpha
}(t,u)du=\lambda t^{\alpha }/\Gamma (1+\alpha ),  %\label{EXP}
\end{equation*}%
where $f_{\alpha }(t,u)$ is given by (\ref{D1}), and 
% Leonenko et al. (
\cite{LMS} showed that 
\begin{equation}\label{Q1:1}
\mathrm{Cov}(N_{\alpha }(t),N_{\alpha }(s)) 
=\frac{\lambda (\min (t,s))^{\alpha }}{\Gamma (1+\alpha )}+\lambda ^{2}%
\mathrm{Cov}(Y_{\alpha }(t),Y_{\alpha }(s)),  
\end{equation}%
where $\mathrm{Cov}(Y_{\alpha }(t),Y_{\alpha }(s))$ is given in \eqref{COV4} while
$\mathrm{Cov}(N(t),N(s))=\lambda \min(t,s)$.
In
particular, 
\begin{equation}\label{eq:VarFPP}
\begin{aligned}
\mathrm{Var}N_{\alpha }(t) &=\lambda ^{2}t^{2\alpha } \Big[\frac{2}{\Gamma
(1+2\alpha )}-\frac{1}{\Gamma ^{2}(1+\alpha )}\Big]+\frac{\lambda t^{\alpha }}{%
\Gamma (1+\alpha )} \\
&=\frac{\lambda ^{2}t^{2\alpha }}{\Gamma ^{2}(1+\alpha )}\Big(\frac{\alpha
\Gamma (\alpha )}{\Gamma (2\alpha )}-1\Big)+\frac{\lambda t^{\alpha }}{\Gamma
(1+\alpha )},\quad t\geq 0.
\end{aligned}%
\end{equation}%
The definition of the Hurst index for renewal processes is
discussed in \cite{DAL}. In the same spirit, one can define the analogous of
the Hurst index for the FPP as 
\begin{equation*}
H=\inf \left\{ \beta :\lim \sup_{T\rightarrow \infty }\frac{\mathrm{Var}%
N_{\alpha }(T)}{{T^{2\beta}}} <\infty\right\} \in(0,1) .
\end{equation*}%
\end{remark}
To prove the formula \eqref{Q1:1}, one can use the  conditional covariance formula \cite[Exercise~7.20.b]{Ross_02}:
\[
\mathrm{Cov}(Z_1,Z_2) = \mathrm{E} \big( \mathrm{Cov}(Z_1,Z_2 | Y) \big) +
\mathrm{Cov} \big( \mathrm{E} (Z_1 | Y) ,\mathrm{E} (Z_2 | Y) \big),
\]
where $Z_1,Z_2$ and $Y$ are random variables, and
\[
\mathrm{Cov}(Z_1,Z_2 | Y)  = \mathrm{E} \big( (Z_1-\mathrm{E} (Z_1 | Y))
(Z_2-\mathrm{E} (Z_2 | Y)) \big) .
\]
Really, if
\[
G_{t,s} (u,v) = \mathrm{P}\{ Y_\alpha (t) \leq u, Y_\alpha (s) \leq v \},
\]
then $\mathrm{E} (N(Y_\alpha(t))|Y_\alpha(t)) = \mathrm{E} (N(1)) \cdot Y_\alpha(t) = \lambda Y_\alpha(t) $, and
\begin{multline*}
\mathrm{Cov}(Y_\alpha (t) , Y_\alpha (s) ) = \mathrm{Var} \Big(N(1) \int_0^\infty\int_0^\infty
\min(u,v)  G_{t,s} (du,dv) \Big) + \mathrm{Cov}\big( \lambda Y_\alpha (t) , \lambda Y_\alpha (s) \big)
\\
= \lambda \mathrm{E} ( Y_\alpha (\min(t,s)) ) +  \lambda^2 \mathrm{Cov}( Y_\alpha (t) , Y_\alpha (s) ),
\end{multline*}
since, for example, if $s\leq t$, then $v=Y_\alpha (s) \leq Y_\alpha (t)=u $, and
\[
\int_0^\infty\int_0^\infty v  G_{t,s} (du,dv) =
\int_0^\infty v \int_0^\infty  G_{t,s} (du,dv) =
\int_0^\infty v \, d\mathrm{P}\{ Y_\alpha (s) \leq v \} = \mathrm{E} ( Y_\alpha (s) ).
\]
\begin{remark}
For more than one random variable in the condition, the conditional covariance formula becomes more complicated,
it can be seen even for the conditional variance formula:
\[
\mathrm{Var}(Z ) =    \mathrm{E} \big( \mathrm{Var}(Z | Y_1,Y_2) \big) +
\mathrm{E} \big( \mathrm{Var} [ \mathrm{E} (Z | Y_1,Y_2) ] | Y_1 \big) +
\mathrm{Var}\big( \mathrm{E} (Z | Y_1) \big).
\]
The corresponding formulas can be found in \cite{BroSwa}. That is why for random fields we develop another technique,
see Appendix.
\end{remark}

\subsection{Watanabe characterization}

Let $\left( \Omega ,\mathcal{F},\mathrm{P}\right) $ be a complete probability
space. Recall that the $\mathcal{F}_{t}-$adapted,
$\mathrm{P}$-integrable stochastic process $%
M=\{M(t),t\geq 0\}$ is an $\mathcal{F}_{t}-$martingale (sub-martingale) if $%
\mathrm{E}(M(t)|\mathcal{F}_{s})=(\geq )M(s),$ $0\leq s\leq t,$ a.s., where $%
\{\mathcal{F}_{t}\}$ is a non-decreasing family of sub-sigma fields of $%
\mathcal{F}$. A point process $N$ is called simple if its jumps are of magnitude $+1$.
It is locally finite when it does not have infinite jumps in a bounded region.
The following theorem is known as the Watanabe
characterization for homogeneous Poisson processes (see, \cite{WAT} 
and \cite[p.~25]{BR}):

\begin{theorem}
Let $N=\{N(t),t\geq 0\}$ be a $\mathcal{F}_{t}-$adapted,
simple locally finite point process. Then $N$%
~is a homogeneous Poisson process iff there is a constant $\lambda >0$, such
that the process $M(t)=N(t)-\lambda t$ is an $\mathcal{F}_{t}-$martingale.
\end{theorem}

We extend the well-known Watanabe characterization for FPP.
The following result may be seen as a corollary of the 
Watanabe characterization for Cox processes as 
in \cite[Chapetr II]{BR}. We will make use of the following lemma.
\begin{lemma}[Doob's Optional Sampling Theorem]\label{lem:DOST}
Let $M$ be a right-continuous martingale.
Then, if $T$ and $S$ are stopping times 
such that $P(T<+\infty)=1$ and $\{M(t\wedge T),t\geq 0\}$ is uniformly integrable,
then $E(M(T)|\mathcal{F}_{S\wedge T})= M(S\wedge T)$.
\end{lemma}
\begin{proof}
Define $N= \{N(t)=M(t\wedge T),t\geq 0\}$. 
Then $N$ is a right-continuous uniformly integrable martingale such that $\lim_{t\to+\infty}N(t)= M(T)$.
Moreover, $N(S) = M(T\wedge S)$. The thesis is hence a consequence of the 
Doob's Optional Sampling Theorem 
 (see, e.g., \cite[Theorem~7.29]{KALL1} with  $X = N$, $\tau \equiv +\infty$ and $\sigma = S$).
\end{proof}

\begin{theorem}\label{thm:Teo2}
Let $X=\{X(t),\,t\geq 0\}$ be a simple locally finite point process. Then $%
X$ is a FPP iff there exist a constant $\lambda >0,$ and an $\alpha $-stable
subordinator $L_{\alpha }=\{L_{\alpha }(t),\,t\geq 0\},$ $0<\alpha <1,$ such
that, denoted by 
$Y_{\alpha }(t)=\inf \{s:L_{\alpha }(s)\geq t\}$ its inverse stable subordinator,
the process
\begin{equation*}
M=\{M(t),\,t\geq 0\}=\{X(t)-\lambda Y_{\alpha }(t),\,t\geq 0\}
\end{equation*}
is a right-continuous martingale with respect to the induced filtration $\mathcal{F}_t=
\sigma(X(s),s\leq t) \vee \sigma(Y_{\alpha }(s),s \geq 0)$ such that, for any $T>0$,
\begin{equation}\label{eq:unBound}
\{M(\tau), \tau\text{ stopping time s.t.\ }Y_{\alpha }(\tau)\leq T\}
\end{equation}
is uniformly integrable. 
\end{theorem}

\begin{proof}
If $X$ is a FPP, then $X(t)=N(Y_{\alpha }(t)),$ where $Y_{\alpha }$ is the
inverse of an $\alpha $-stable subordinator and $N$ is a Poisson process
with intensity $\lambda >0.$

Note that $X\geq 0$ and $(Y_{\alpha }\geq 0$ are monotone non-decreasing, and hence the boundenesses in $L^2$ given by \eqref{eq:VarFPP}
and Proposition~\ref{prop:2.1} iiiA) 
imply that $\{N(Y_{\alpha }(t))-\lambda Y_{\alpha }(t), 0\leq t \leq T \}$
is uniformly integrable (see, for example, \cite[pag.~67]{KALL1}).
Therefore $N(Y_{\alpha }(t))-\lambda Y_{\alpha }(t)$ is still a martingale, by Lemma~\ref{lem:DOST}. Notice
that $Y_{\alpha }(t)$ is continuous increasing and adapted; therefore it is
the predictable intensity of the sub-martingale $X.$ 

Now, let $\tau$ be a stopping time s.t.\ $Y_{\alpha }(\tau)\leq T$, and hence $\lambda Y_{\alpha }(\tau)\leq \lambda T$. 
Then, since $N$ is a Poisson process
with intensity $\lambda >0$,
$\tilde{M}(t) = M(\tau\wedge t)$ is a martingale bounded in $L^2$ and null at $0$, and therefore
it converges in $L^2$ to $M(\tau)$, with variance bounded by 
\[
E(M^2(\tau)) = \lim_{t\to\infty} E(M^2(\tau\wedge t)) \leq \mathrm{Var}(N(T)) + 
\mathrm{Var}(Y_{\alpha }(\tau)) \leq 
\text{const} \cdot (1+T^2).
\]
Then the family \eqref{eq:unBound} is uniformly bounded in $L^2$, which implies the thesis.

Conversely, it is enough to prove that $X(t)=N(Y_{\alpha }(t)),$ where $N$
is a Poisson process, independent of $Y_{\alpha }.$

Consider the inverse of $Y_{\alpha }(t):$ 
\begin{equation*}
Z(t)=\inf \{s:\,Y_{\alpha }(s)\geq t\}.
\end{equation*}%
%Notice that in general $Z(t)$ does not equal to $L_{\alpha }(t),$ since for
%example $Z(t)$ is continuous, but $L_{\alpha }(t)$ is not.

$\{Z(t),\,t\geq 0\}$ can be seen as a family of stopping times. Then, by Lemma~\ref{lem:DOST},
\begin{equation*}
M(Z(t))=X(Z(t))-\lambda Y_{\alpha }(Z(t))
\end{equation*}%
is still a martingale. The fact that $Y_{\alpha }$ is continuous implies that $Y_{\alpha }(Z(t))=t$,
and hence $X(Z(t))-\lambda t$ is a martingale. Moreover, since $Z(t)$
is increasing, $X(Z(t))$ is a simple point process.

Following the classical Watanabe characterization, $X(Z(t))$ is a classical
Poisson process with parameter $\lambda >0$. Call this process $%
N(t)=X(Z(t)). $ Then $X(t)=N(Y_{\alpha }(t))$ is a FPP.
\end{proof}

For recent developments and random change time results, 
see also \cite{Magdz10,Nane17}. 
In particular, we thank a referee to have outlined that a similar result 
has been obtained in \cite[Lemma 3.2]{Nane17}.

\section{Mixed-Fractional Poisson Processes}\label{sect:6}

\subsection{Definition}

In this section, we consider a more general Mixed-Fractional Poisson process
(MFPP) 
\begin{equation}
N^{\alpha _{1},\alpha _{2}}=\{N^{\alpha _{1},\alpha _{2}}(t),t\geq
0\}=\{N(Y_{\alpha _{1},\alpha _{2}}(t)),t\geq 0\},  \label{SR}
\end{equation}%
where the homogeneous Poisson process $\ N$ with intensity $\lambda >0,$ and
the inverse subordinator $Y_{\alpha _{1},\alpha _{2}}$ given by (\ref{ISM})
are independent. We will show that $N^{\alpha _{1},\alpha _{2}}$ is the
stochastic solution of the system of fractional differential-difference
equations: for $k=0,1,2,\ldots$,

\begin{equation}
C_{1}\mathrm{D}_{t}^{\alpha _{1}}p_{k}^{(\alpha _{1},\alpha _{2})}(t)+C_{2}%
\mathrm{D}_{t}^{\alpha _{2}}p_{k}^{(\alpha _{1},\alpha _{2})}(t)=-\lambda
(p_{k}^{(\alpha _{1},\alpha _{2})}(t)-p_{k-1}^{(\alpha _{1},\alpha
_{2})}(t)),  \label{ME1}
\end{equation}%
with initial conditions: 
\begin{equation}
p_{0}^{(\alpha _{1},\alpha _{2})}(0)=1,p_{k}^{(\alpha _{1},\alpha
_{2})}(0)=0,p_{-1}^{(\alpha _{1},\alpha _{2})}(t)=0,\quad k\geq 1,
\label{ME2}
\end{equation}%
where $\mathrm{D}_{t}^{\alpha }$ is the fractional Caputo-Djrbashian
derivative (\ref{FD}), and for $C_{1}\geq 0,C_{2}>0,C_{1}+C_{2}=1$, $\alpha
_{1},\alpha _{2}\in (0,1)$,
\begin{equation*}
p_{k}^{(\alpha _{1},\alpha _{2})}(t)=\mathrm{P}\{N^{\alpha _{1},\alpha
_{2}}(t)=k\},\qquad k=0,1,2\ldots
\end{equation*}

\subsection{Distribution Properties}

Using the formulae for Laplace transform of the fractional Caputo-Djrbashian
derivative (see, \cite[p.39]{MS1}):%
\begin{equation*}
\int_{0}^{\infty }e^{-st}\mathrm{D}_{t}^{\alpha }u(t)dt=s^{\alpha }
%\tilde{u}(s)
u(0^+)
-s^{\alpha -1}u(0),0<\alpha <1,
\end{equation*}%
one can obtain from (\ref{ME1}) with $k=0$ the following equation 
\begin{equation*}
C_{1}s^{\alpha _{1}}\tilde{p}_{0}(s)-C_{1}s^{\alpha _{1}-1}+C_{2}s^{\alpha 2}%
\tilde{p}_{0}(s)-C_{2}s^{\alpha _{2}-1}=-\lambda \tilde{p}_{0}(s),\tilde{p}%
_{0}(0)=1,
\end{equation*}%
for the Laplace transform 
\begin{equation*}
\tilde{p}_{0}^{(\alpha _{1},\alpha _{2})}(s)=\tilde{p}_{0}(s)=\int_{0}^{%
\infty }e^{-st}p_{0}^{(\alpha _{1},\alpha _{2})}(t)dt,\quad s\geq 0.
\end{equation*}

Thus 
\begin{equation*}
\tilde{p}_{0}(s)=\frac{C_{1}s^{\alpha _{1}-1}+C_{2}s^{\alpha _{2}-1}}{%
\lambda +C_{1}s^{\alpha _{1}}+C_{2}s^{\alpha _{2}}},\quad s\geq 0,
\end{equation*}%
and using the formula for an inverse Laplace transform (see, \cite{HMS}), for
$\Re \alpha > 0,\Re \beta > 0, \Re s> 0, \Re (\alpha-\rho) > 0,\Re (\alpha-\beta) > 0,$ and $|as^{\beta}/(s^{\alpha}+b)|<1$:
\begin{equation}\label{eq:InvLap}
\mathcal{L}^{-1}\Big(\frac{ s^{\rho-1}}{ s^{\alpha}+a s^{\beta}+b};t\Big)
= t^{\alpha-\rho} \sum_{r=0}^\infty (-a)^r t^{(\alpha-\beta)r} E_{\alpha,\alpha+(\alpha-\beta)r-\rho+1}^{r+1} (-b t^{\alpha}), 
\end{equation}
one can find an exact form of the $p_{0}^{(\alpha _{1},\alpha _{2})}(t)$ in
terms of generalized Mittag-Leffler functions (\ref{ML3}): 
\begin{align}
p_{0}^{(\alpha _{1},\alpha _{2})}(t)& =\mathop{\displaystyle \sum}%
\limits_{r=0}^{\infty }\left( -\frac{C_{1}}{C_{2}}t^{\alpha _{2}-\alpha
_{1}}\right) ^{r}E_{\alpha _{2},(\alpha _{2}-\alpha _{1})r+1}^{r+1}\left( -%
\frac{\lambda }{C_{2}}t^{\alpha _{2}}\right)  \label{F} \\
& \quad -\mathop{\displaystyle \sum}\limits_{r=0}^{\infty }\left( -\frac{%
C_{1}}{C_{2}}t^{\alpha _{2}-\alpha _{1}}\right) ^{r+1}E_{\alpha _{2},(\alpha
_{2}-\alpha _{1})(r+1)+1}^{r+1}\left( -\frac{\lambda }{C_{2}}t^{\alpha
_{2}}\right) .  \notag
\end{align}%
For $k\geq 1,$we obtain from (\ref{ME1}): 
\begin{equation*}
\tilde{p}_{k}(s)(\lambda +C_{1}s^{\alpha _{1}}+C_{2}s^{\alpha _{2}})=\lambda 
\tilde{p}_{k-1}(s),
\end{equation*}%
where 
\begin{equation*}
\tilde{p}_{k}^{(\alpha _{1},\alpha _{2})}(s)=\tilde{p}_{k}(s)=\int_{0}^{%
\infty }e^{-st}p_{k}^{(\alpha _{1},\alpha _{2})}(t)dt,\quad s\geq 0.
\end{equation*}%
Thus from (\ref{ME1}) we obtain the following expression for the Laplace
transform of $p_{k}^{(\alpha _{1},\alpha _{2})}(t),$ $k\geq 0:$ 
\begin{align}
\tilde{p}_{k}(s)& =\left( \frac{\lambda }{\lambda +C_{1}s^{\alpha
_{1}}+C_{2}s^{\alpha _{2}}}\right) \tilde{p}_{k-1}(s)=\left( \frac{\lambda }{%
\lambda +C_{1}s^{\alpha _{1}}+C_{2}s^{\alpha _{2}}}\right) ^{k}\tilde{p}%
_{0}(s)  \label{FOR1} \\
& =\frac{\lambda ^{k}(C_{1}s^{\alpha _{1}-1}+C_{2}s^{\alpha _{2}-1})}{%
(\lambda +C_{1}s^{\alpha _{1}}+C_{2}s^{\alpha _{2}})^{k+1}}=\frac{\lambda
^{k}(C_{1}s^{\alpha _{1}}+C_{2}s^{\alpha _{2}})}{s(\lambda +C_{1}s^{\alpha
_{1}}+C_{2}s^{\alpha _{2}})^{k+1}},\ k=0,1,2...  \notag
\end{align}

\bigskip On the other hand, one can compute the Laplace transform from the
stochastic representation (\ref{SR}). If 
\begin{equation}
p_{k}^{(\alpha _{1},\alpha _{2})}(t)=\mathrm{P}\{N(Y_{\alpha _{1},\alpha
_{2}}(t))=k\}=\int_{0}^{\infty }\frac{e^{-\lambda x}}{k!}(\lambda
x)^{k}f_{\alpha _{1},\alpha _{2}}(t,x)dx,  \label{MD3}
\end{equation}%
where $f_{\alpha _{1},\alpha _{2}}(t,x)$ is given by (\ref{MD2}), then using
(\ref{M2}),(\ref{M3}) we have for $k\geq 0,s>0$%
\begin{equation*}
\tilde{p}_{k}(s)=\int_{0}^{\infty }e^{-st}p_{k}^{(\alpha _{1},\alpha
_{2})}(t)dt=\int_{0}^{\infty }\frac{e^{-\lambda x}}{k!}(\lambda
x)^{k}\Big[\int_{0}^{\infty }e^{-st}f_{\alpha _{1},\alpha _{2}}(t,x)dt\Big]dx
\end{equation*}%
\begin{equation*}
=\frac{\lambda ^{k}}{k!}\frac{\phi (s)}{s}\int_{0}^{\infty }e^{-\lambda
x}x^{k}e^{-x\phi (s)}dx
\end{equation*}%
Note that 
\begin{align*}
&\frac{\partial ^{k}}{\partial \lambda ^{k}}\int_{0}^{\infty }e^{-\lambda
x}e^{-x\phi (s)}dx  =(-1)^{k}\int_{0}^{\infty }e^{-\lambda x}x^{k}e^{-x\phi
(s)}dx
\\
& \qquad =\frac{\partial ^{k}}{\partial \lambda ^{k}}\frac{1}{\lambda +\phi (s)}%
=(-1)^{k}\frac{k!}{(\lambda +\phi (s))^{k+1}};
\end{align*}%
thus%
\begin{equation*}
\tilde{p}_{k}(s)=\lambda ^{k}\frac{\phi (s)}{s(\lambda +\phi (s))^{k+1}}=%
\frac{\lambda ^{k}(C_{1}s^{\alpha _{1}}+C_{2}s^{\alpha _{2}})}{s(\lambda
+C_{1}s^{\alpha _{1}}+C_{2}s^{\alpha _{2}})^{k+1}},
\end{equation*}%
the same expression as (\ref{FOR1}). We can formulate the result in the
following form:

\begin{theorem}
The MFPP $N^{\alpha _{1},\alpha _{2}}$ defined in (\ref{SR}) is the
stochastic solution of the system of fractional differential-difference
equations (\ref{ME1}) with initial conditions (\ref{ME2}).
\end{theorem}

Note that in \cite{BEG} one can find some other stochastic representations
of the MFPP (\ref{SR}). Also, some analytical expression for $p_{0}^{(\alpha
_{1},\alpha _{2})}(t)$ is given by (\ref{F}), while the analytical
expression for $p_{k}^{(\alpha _{1},\alpha _{2})}(t),$for $k\geq 1,$ are
given by (\ref{MD3}).

Moreover, $p_{k}^{(\alpha _{1},\alpha _{2})}(t),$for $k\geq 1,$ can be obtained
by the following recurrent relation: 
\begin{equation*}
p_{k}^{(\alpha _{1},\alpha _{2})}(t)=\mathop{\displaystyle \int}%
\limits_{0}^{t}p_{k-1}^{(\alpha _{1},\alpha _{2})}(t-z)g(z)dz,
\end{equation*}%
where 
\begin{equation*}
\tilde{g}(s)=\int_{0}^{\infty }e^{-sz}g(z)dz=\frac{\lambda }{\lambda
+C_{1}s^{\alpha _{1}}+C_{2}s^{\alpha _{2}}},
\end{equation*}
and from \eqref{eq:InvLap}:
\[
g(z) =\frac{\lambda}{C_2} z^{\alpha_2-1} \sum_{r=0}^\infty \Big(-\frac{C_1}{C_2}z^{\alpha_2-\alpha_1} \Big)^r % t^{(\alpha-\beta)r}
E_{\alpha_2,\alpha_2+(\alpha_2-\alpha_1)r}^{r+1} \Big(-\frac{\lambda}{C_2} z^{\alpha_2}\Big).
\]

\subsection{Dependence}

From \cite[Theorem 2.1]{LMS} and (\ref{UI}), we have the following
expressions for moments in form of the function%
\begin{align*}
U(t) &=\frac{1}{C_{2}}t^{\alpha _{2}}E_{\alpha _{2}-\alpha _{1},\alpha
_{2}+1}(-C_{1}t^{_{\alpha _{2}-\alpha _{1}}}/C_{2}) ,\\
\mathrm{E}N^{\alpha _{1},\alpha _{2}}(t)&=\lambda U(t), \\
\mathrm{Var}N^{\alpha _{1},\alpha _{2}}(t) &=\lambda ^{2}\frac{1}{C_{2}^{2}}%
t^{2\alpha _{2}}[2E_{\alpha _{2}-\alpha _{1},\alpha _{1}+\alpha
_{2}+1}(-C_{1}t^{_{\alpha _{2}-\alpha _{1}}}/C_{2}) \\
&\qquad -(E_{\alpha _{2}-\alpha _{1},\alpha _{2}+1}(-C_{1}t^{_{\alpha _{2}-\alpha
_{1}}}/C_{2}))^{2}]\\
&\qquad +\lambda \frac{1}{C_{2}}t^{\alpha _{2}}E_{\alpha
_{2}-\alpha _{1},\alpha _{2}+1}(-C_{1}t^{_{\alpha _{2}-\alpha _{1}}}/C_{2}),
\\
\mathrm{Cov}(N^{\alpha _{1},\alpha _{2}}(t),N^{\alpha _{1},\alpha _{2}}(s))
& =\lambda U(\min (t,s)) 
+\lambda ^{2}\Big\{ \int_{0}^{\min
(t,s)}\Big(U(t-\tau )\\
& \qquad \qquad +U(s-\tau )\Big)dU(\tau )-U(t)U(s)\Big\} .
\end{align*}

We extend the Watanabe characterization for MFPP. 
Let $\Lambda (t):\mathbb R_+ \to \mathbb R_+ $ be a non-negative right-continuous 
non-decreasing deterministic function such that $\Lambda(0)=0$, %$\lim_{t\to +\infty}\Lambda(t)=+\infty$,
$\Lambda(\infty)=\infty$,
and $\Lambda(t)-\Lambda(t-)\leq 1$ for
any $t$. Such a function will be called \emph{consistent}.
The Mixed-Fractional Non-homogeneous Poisson process
(MFNPP) is defined as
\begin{equation*}
N^{\alpha _{1},\alpha _{2}}_{\Lambda}=\{N^{\alpha _{1},\alpha _{2}}_{\Lambda}(t),t\geq
0\}= \{N(\Lambda(Y_{\alpha _{1},\alpha _{2}}(t))),t\geq 0\}, % \label{SRN}
\end{equation*}%
where the homogeneous Poisson process $\ N$ with intensity $\lambda = 1$, and
the inverse subordinator $Y_{\alpha _{1},\alpha _{2}}$ given by (\ref{ISM})
are independent. 

\begin{theorem}
Let $X=\{X(t),\,t\geq 0\}$ be a simple locally finite point process. $X$
is a MFNPP iff there exist a consistent function $ \Lambda(t),$ and a mixed stable
subordinator \{$L_{\alpha _{1},\alpha _{2}}(t),\,t\geq 0\},$ $0<\alpha
_{1}<1,$ $0<\alpha _{2}<1,$ defined in (\ref{M1}), such that 
\begin{equation*}
M=\{M(t),\,t\geq 0\}=\{X(t)-\Lambda(Y_{\alpha _{1},\alpha _{2}}(t)),\,t\geq
0\}
\end{equation*}%
is a martingale with respect to the induced filtration $\mathcal{F}_t = 
\sigma( X(s), s\leq t) \vee \sigma(Y_{\alpha _{1},\alpha _{2}}(s),s\geq 0)$, where $Y_{\alpha _{1},\alpha _{2}}(t)=\inf \{s:L_{\alpha
_{1},\alpha _{2}}(t)\geq t\}$ is the inverse mixed stable subordinator. 
In addition, for any $T>0$,
\begin{equation*}%\label{eq:unBound}
\{M(\tau), \tau\text{ stopping time s.t.\ }\Lambda(Y_{\alpha _{1},\alpha _{2}}(\tau))\leq T\}
\end{equation*}
is uniformly integrable. 
\end{theorem}

%\begin{proof}
%If $X$ is a MFNPP, then $X(t)=N(\Lambda(Y_{\alpha _{1},\alpha _{2}}(t))),$ 
%where the homogeneous Poisson process $\ N$ and
%the inverse subordinator $Y_{\alpha _{1},\alpha _{2}}$ given by (\ref{ISM})
%are independent.
%Therefore $X(t)-\Lambda(Y_{\alpha _{1},\alpha _{2}}(t))$ is a martingale. 
%%Notice
%%that $\Lambda(Y_{\alpha _{1},\alpha _{2}}(t))$ is continuous increasing and adapted; therefore it is
%%the predictable compensator of the sub-martingale $X.$
%
%Conversely, it is enough to prove that $X(t)=N(\Lambda(Y_{\alpha _{1},\alpha _{2}}(t))),$ where $N$
%is a Poisson process, independent of $Y_{\alpha }.$
%
%Consider the pseudo inverse of $Y_{\alpha _{1},\alpha _{2}}(t):$ 
%\begin{equation*}
%Z(t)=\inf \{s:\,Y_{\alpha _{1},\alpha _{2}}(s)\geq t\}.
%\end{equation*}%
%
%$\{Z(t),\,t\geq 0\}$ can be seen as a family of stopping times. Then, by the
%Doob optional stopping theorem, 
%\begin{equation*}
%M(Z(t))=X(Z(t))-\Lambda(Y_{\alpha _{1},\alpha _{2}}(Z(t)))
%\end{equation*}%
%is still a martingale. Then $X(Z(t))-\Lambda(t)$ is a martingale. Since $Z(t)$
%is increasing, $X(Z(t))$ is a simple point process.
%
%Following the extension of Watanabe characterization given in \cite{BR2}, $X(Z(t))$ is a non-homogeneous
%Poisson process with deterministic intensity $\Lambda(t)$. Call this process $%
%N(t)=X(Z(t)). $ Then $X(t)=\Lambda(Y_{\alpha _{1},\alpha _{2}}(Z(t)))$ is a MFNPP.
%\qed \end{proof}

\begin{proof}
The proof is analogue to that of Theorem~\ref{thm:Teo2}.
\end{proof}

\section{Two-Parameter Fractional Poisson Processes and Martingales}\label{sect:4}

\subsection{Homogeneous Poisson random fields}

This section collects some known results from the theory of two-parameter
Poisson processes and homogeneous Poisson random fields (PRF) (see, e.g., 
\cite{SKM,MN}, among the others).

Let $\left( \Omega ,\mathcal{F},P\right) $ be a complete probability space
and let $\left\{ \mathcal{F}_{t_{1},t_{2}};\left( t_{1},t_{2}\right) \in 
\mathbb{R}_+^2\right\} $ be a family of sub-$\sigma $-fields of $\mathcal{F}$
such that

(i) $\mathcal{F}_{s_{1},s_{2}}\subseteq \mathcal{F}_{t_{1},t_{2}}$ for any $%
\quad s_{1}\leq t_{1},\quad s_{2}\leq t_{2};$

(ii) $\mathcal{F}_{0,0}$ contains all null sets of $\mathcal{F};$

(iii) for each $z\in \mathbb{R}_+^2,$ $\mathcal{F}_{z}=\mathop{%
\displaystyle \bigcap }\limits_{z\prec z^{\prime }}\mathcal{F}_{z^{\prime }} 
$ where $z=\left( s_{1},s_{2}\right) \prec z^{\prime }=\left(
t_{1},t_{2}\right) $ denotes the partial order on $\mathbb{R}_+^2,$ which
means that $s_{1}\leq t_{1},\quad s_{2}\leq t_{2}.$

Given $\left( s_{1},s_{2}\right) \prec \left( t_{1},t_{2}\right) $ we denote
by 
\begin{equation*}
\Delta_{s_{1},s_{2}}X(t_{1},t_{2})=X(t_{1},t_{2})-X(t_{1},s_{2})-X(s_{1},t_{2})+X(s_{1},s_{2})
%\label{2.1}
\end{equation*}
the increments of the random field $X(t_{1},t_{2}),(t_{1},t_{2})\in \mathbb{R%
}_{+}^{2}$ over the rectangle $(\left( s_{1},s_{2}\right) ,\left(
t_{1},t_{2}\right) ]$.
In addition, we denote
\begin{equation*}
\mathcal{F}_{\infty ,t_{2}}=% \mathop{\bigvee }_{t_{1}>0}%
\sigma( \mathcal{F}_{t_{1},t_{2}}, t_{1}>0 ) , 
\mathcal{F}_{t_{1},\infty }=% \mathop{\bigvee }_{t_{1}>0}%
\sigma( \mathcal{F}_{t_{1},t_{2}}, t_{2}>0 ) , 
\text{ and }\mathcal{F}^*_{s_1,s_2} = \mathcal{F}_{\infty,s_2} \vee \mathcal{F}_{s_1, \infty}
=%
\sigma(  \mathcal{F}_{s_{1},\infty },\mathcal{F}_{\infty ,s_{2}}).
\end{equation*}

A strong martingale is an integrable two-parameter process $X$ such that 
$$ E (\Delta_{s_{1},s_{2}}X(t_{1},t_{2}) | \mathcal{F}_{\infty,s_2} \vee \mathcal{F}_{s_1, \infty}) = 0,$$
for any $z=\left( s_{1},s_{2}\right) \prec z^{\prime }=\left(
t_{1},t_{2}\right) \in \mathbb{R}_+^2$.

Let $\left\{ \mathcal{F}_{t_{1},t_{2}}\right\} $ be a family of sub-$\sigma $%
-fields of $\mathcal{F}$ satisfying the previous conditions (i), (ii), (iii)
for all $(t_{1},t_{2})\in \mathbb{R}_{+}^{2}.$ A $\mathcal{F}%
_{t_{1},t_{2}}- $ PRF is an adapted, cadlag field $N=\left\{
N(t_{1},t_{2}),(t_{1},t_{2})\in \mathbb{R}_{+}^{2}\right\} $, such that,

(1) $N(t_{1},0)=N(0,t_{2})=0$ a.s.

(2) for all $\left( s_{1},s_{2}\right) \prec \left( t_{1},t_{2}\right) $ the
increments $\Delta _{s_{1},s_{2}}N(t_{1},t_{2})$ are independent of $%
\mathcal{F}_{\infty ,s_{2}}\mathop{\vee } \mathcal{F}%
_{s_{1},\infty },$ and has a Poisson law with parameter $\lambda \left(
t_{1}-s_{1}\right) \left( t_{2}-s_{2}\right) $, that is, 
\begin{equation*}
\mathrm{P}\left\{ \Delta _{s_{1},s_{2}}N(t_{1},t_{2})=k\right\} =\frac{%
e^{-\lambda \left\vert S\right\vert }\left( \lambda \left\vert S\right\vert
\right) ^{k}}{k!},\quad \lambda >0,\quad k=0,1,\ldots ,
\end{equation*}
where $S=(\left( s_{1},s_{2}\right) ,\left( t_{1},t_{2}\right) ]$, 
$\lambda >0$, and $|S|$ is the Lebesgue measure of $S$.

If we do not specify the filtration, $\left\{ \mathcal{F}_{t_{1},t_{2}}%
\right\} $ will be the filtration generated by the field itself, completed
with the nulls sets of $\mathcal{F}^{N}=\sigma \left\{
N(t_{1},t_{2}),(t_{1},t_{2})\in \mathbb{R}_{+}^{2}\right\} .$

It is known that then there is a simple locally finite point random measure $%
N(\cdot ),$ such that for any finite $n=1,2,\dots ,$ and for any disjoint
bounded Borel sets $A_{1},...,A_{n}$%
\begin{align*}
&\mathrm{P}\left( N(A_{1})=k_{1},...,N(A_{n})=k_{n}\right)  %\label{2.2} 
\\
& \qquad =\frac{\lambda ^{k_{1}+...+k_{n}}}{k_{1}!\cdot ..\cdot k_{n}!}(\left\vert
A_{1}\right\vert )^{k_1}\cdot \cdot \cdot (\left\vert A_{n}\right\vert
)^{k_{n}}\exp \left\{ -\sum_{j=1}^{n}\lambda \left\vert A_{j}\right\vert
\right\} ,\quad k_{j}=0,1,2,...,  \notag
\end{align*}
while 
\begin{equation*}
\mathrm{E}N(A)=\lambda \left\vert A\right\vert ,\text{ }\mathrm{Cov}%
(N(A_{1}),N(A_{2}))=\lambda \left\vert A_{1}\cap A_{2}\right\vert .
\end{equation*}

\begin{theorem}[Two Parameter Watanabe Theorem \cite{IM3}]
A random simple locally finite counting measure $N$ is a two-parameter PRF 
iff $N(t_1,t_2)-\lambda t_1t_2$ is a strong martingale.
\end{theorem}

\subsection{Fractional Poisson random fields}

Let $Y_{\alpha _{1}}^{(1)}(t),t\geq 0$ and $Y_{\alpha _{2}}^{(2)}(t),t\geq 0$
be two independent inverse stable subordinators with indices $\alpha _{1}\in
(0,1)$ and $\alpha _{2}\in (0,1),$ which are independent of the Poisson
field $N(t_{1},t_{2}),(t_{1},t_{2})\in \mathbb{R}_{+}^{2}.$ 
%Leonenko and Merzbach 
In \cite{LM}, the Fractional Poisson field (FPRF) is defined as follows
\begin{equation}
N_{\alpha _{1},\alpha _{2}}(t_{1},t_{2})=N(Y_{\alpha
_{1}}^{(1)}(t_1),Y_{\alpha _{2}}^{(2)}(t_{2})),\quad (t_{1},t_{2})\in 
\mathbb{R}_{+}^{2}.  \label{FPRF}
\end{equation}

We obtain the marginal distribution of FPRF: for $k=0,1,2,\ldots$,
\begin{align}
p_{k}^{\alpha _{1},\alpha _{2}}(t_{1},t_{2})& =\mathrm{P}\left( N_{\alpha
_{1},\alpha _{2}}(t_{1},t_{2})=k\right) \notag \\
& =\int_{0}^{\infty }\int_{0}^{\infty }\frac{e^{-\lambda x_{1}x_{2}}(\lambda
x_{1}x_{2})^{k}}{k!}f_{\alpha _{1}}(t_{1},x_{1})f_{\alpha
_{2}}(t_{2},x_{2})dx_{1}dx_{2},  \label{P2}
\end{align}%
where $f_{a}(t,x)$ is given by (\ref{D1}). In other words, for $(t_{1},t_{2})\in \mathbb{R}_{+}^{2},\ k=0,1,\ldots  $
\begin{align}
\mathrm{P}\left( N_{\alpha _{1},\alpha _{2}}(t_{1},t_{2})=k\right)
&
=\frac{t_{1}t_{2}\lambda ^{k}}{\alpha _{1}\alpha _{2}k!}\int_{0}^{\infty
}\int_{0}^{\infty }e^{-\lambda x_{1}x_{2}}x_{1}^{k-1-\frac{1}{\alpha _{1}}%
}x_{2}^{k-1-\frac{1}{\alpha _{2}}}g_{\alpha _{1}}(t_{1}x_{1}^{-\frac{1}{%
\alpha _{1}}})g_{\alpha _{2}}(t_{2}x_{2}^{-\frac{1}{\alpha _{2}}%
})dx_{1}dx_{2},  \notag \\
& =\frac{\lambda ^{k}}{k!t_{1}t_{2}}\int_{0}^{\infty }\int_{0}^{\infty
}e^{-\lambda x_{1}x_{2}}x_{1}^{k+\frac{1}{\alpha _{1}}}x_{2}^{k+\frac{1}{%
\alpha _{2}}}W_{-\alpha _{1},0}(-\frac{x_{1}}{t_{1}^{\alpha _{1}}}%
)W_{-\alpha _{2},0}(-\frac{x_{2}}{t_{2}^{\alpha _{2}}})dx_{1}dx_{2},  \label{LP} 
\end{align}%
where the Wright generalized Bessel function is defined by (\ref{WR}), and 
$g_{\alpha}(x)$ is defined by \eqref{10}.

Using the Laplace transform given by (\ref{L1})
one can obtain an exact expression for the double Laplace transform of (\ref%
{P2}): for $k=0,1,2,\ldots$,
\begin{multline}
\mathcal{L}\{p_{k}(t_{1},t_{2});s_{1},s_{2}\}=\int_{0}^{\infty
}\int_{0}^{\infty }e^{-s_{1}t_{1}-s_{2}t_{2}}p_{k}(t_{1},t_{2})dt_{1}dt_{2}
\label{L2} \\
=\int_{0}^{\infty }\int_{0}^{\infty }\frac{e^{-\lambda x_{1}x_{2}}(\lambda
x_{1}x_{2})^{k}}{k!}s_{1}^{\alpha _{1}-1}s_{2}^{\alpha _{2}-1}\exp
\{-x_{1}s_{1}^{\alpha _{1}}-x_{2}s_{2}^{\alpha _{2}}\}dx_{1}dx_{2}.
\end{multline}

%{\color{red}It is known (see, e.g., \cite[Theorems 1.1. and 1.2]{BAD}) that a single
%point random field $M$ can be regarded as a random set of points and could
%be characterized %either 
%by its system of finite-dimensional distributions as 
%\begin{equation*}
%\mathrm{P}\left( M(B_{1})=k_{1},...,M(B_{m})=k_{m}\right)
%\end{equation*}
%for all integers $m>0$ and all compacts $B_{1},..,B_{m}$.}

Note that 
\begin{align}
\mathrm{E}N_{\alpha _{1},\alpha _{2}}(t_{1},t_{2})& =
\mathrm{E}\big[
\mathrm{E}[N(Y_{\alpha _{1}}(t_1),Y_{\alpha _{2}}(t_2)) | 
Y_{\alpha _{1}}(t_1),Y_{\alpha _{2}}(t_2) ] \big]  \notag \\
& =\int_{0}^{\infty
}\int_{0}^{\infty }\mathrm{E}N(u_{1},u_{2})f_{\alpha
_{1}}(t_{1},u_{1})f_{\alpha _{2}}(t_{2},u_{2})du_{1}du_{2}  \notag \\
& =\lambda t_{1}^{\alpha _{1}}t_{2}^{\alpha _{2}}/[\Gamma (1+\alpha
_{1})\Gamma (1+\alpha _{2})]  \label{EXP1}
\end{align}%
and, for $(t_{1},t_{2}),(s_{1},s_{2})\in \mathbb{R}_{+}^{2}$,
\begin{multline}
\mathrm{Cov}(N_{\alpha _{1},\alpha _{2}}(t_{1},t_{2}),N_{\alpha _{1},\alpha
_{2}}(s_{1},s_{2})) \label{COVF} \\\begin{footnotesize}
\begin{aligned}
=\lambda ^{2}&\Big\{\Big[\frac{1}{\Gamma (1+\alpha _{1})\Gamma (\alpha _{1})}%
\int_{0}^{\min (t_{1},s_{1})}(t_{1}-\tau _{1})^{\alpha _{1}}+(s_{1}-\tau
_{1})^{\alpha _{1}})\tau _{1}^{\alpha _{1}-1}d\tau _{1}-\frac{%
(s_{1}t_{1})^{\alpha _{1}}}{\Gamma ^{2}(1+\alpha _{1})}\Big]   \\
&\times \Big[\frac{1}{\Gamma (1+\alpha _{2})\Gamma (\alpha _{2})}%
\int_{0}^{\min (t_{2},s_{2})}(t_{2}-\tau _{2})^{\alpha _{2}}+(s_{2}-\tau
_{2})^{\alpha _{2}})\tau _{2}^{\alpha _{2}-1}d\tau _{2}-\frac{%
(s_{2}t_{2})^{\alpha _{2}}}{\Gamma ^{2}(1+\alpha _{2})}\Big]   \\
&
+\frac{(t_{1}s_{1})^{\alpha _{1}}}{\Gamma ^{2}(1+\alpha _{1})}   
\Big[\frac{1}{%
\Gamma (1+\alpha _{2})\Gamma (\alpha _{2})}\int_{0}^{\min
(t_{2},s_{2})}\left( (t_{2}-\tau _{2})^{\alpha _{2}}+(s_{2}-\tau
_{2})^{\alpha _{2}}\right) \tau _{2}^{\alpha _{2}-1}d\tau _{2}-\frac{%
(s_{2}t_{2})^{\alpha _{2}}}{\Gamma ^{2}(1+\alpha _{2})} \Big]  \\
&
+\frac{(t_{2}s_{2})^{\alpha _{1}}}{\Gamma ^{2}(1+\alpha _{2})}
\Big[\frac{1}{%
\Gamma (1+\alpha _{1})\Gamma (\alpha _{1})}\int_{0}^{\min
(t_{1},s_{1})}\left( (t_{1}-\tau _{1})^{\alpha _{1}}+(s_{1}-\tau
_{1})^{\alpha _{1}}\right) \tau _{1}^{\alpha _{1}-1}d\tau _{1}-\frac{%
(s_{1}t_{1})^{\alpha _{1}}}{\Gamma ^{2}(1+\alpha _{1})} \Big] \Big\} \\
& +\lambda \frac{(\min (t_{1},s_{1}))^{\alpha _{1}}(\min
(t_{2},s_{2}))^{\alpha _{2}}}{\Gamma (1+\alpha _{1})\Gamma (1+\alpha _{2})}%
;
\end{aligned} \end{footnotesize}
\end{multline}
in particular, for $(t_{1},t_{2}),(s_{1},s_{2})\in \mathbb{R}_{+}^{2}$,
\begin{equation}
\mathrm{Var}N_{\alpha _{1},\alpha _{2}}(t_{1},t_{2})=\lambda
^{2}t_{1}{}^{2\alpha _{1}}t_{2}{}^{2\alpha _{2}}C_{1}(\alpha _{1},\alpha
_{2})+\lambda t_{1}{}^{\alpha _{1}}t_{2}{}^{\alpha _{2}}C_{2}(\alpha
_{1},\alpha _{2})\},  \label{VAR}
\end{equation}%
where 
\begin{eqnarray*}
C_{1}(\alpha _{1},\alpha _{2}) &=&\frac{1}{\alpha _{1}\alpha _{2}\Gamma
(2\alpha _{1})\Gamma (2\alpha _{2})}-\frac{1}{(\alpha _{1}\alpha
_{2})^{2}\Gamma ^{2}(\alpha _{1})\Gamma ^{2}(\alpha _{2})};\quad \\
C_{2}(\alpha _{1},\alpha _{2}) &=&\frac{1}{\Gamma (1+\alpha _{1})\Gamma
(1+\alpha _{2})}.
\end{eqnarray*}

We can summarize our results in the following

\begin{proposition}\label{prop:4.2}
Let $N_{\alpha _{1},\alpha
_{2}}(t_{1},t_{2}),(t_{1},t_{2})\in \mathbb{R}_{+}^{2},$ be a FPRF defined
by (\ref{FPRF}). Then
\begin{enumerate}
\item[i)] $\mathrm{P}\left( N_{\alpha _{1},\alpha
_{2}}(t_{1},t_{2})=k\right) ,$ $k=0,1,2...$ is given by (\ref{LP});
\item[ii)] $\mathrm{E}N_{\alpha _{1},\alpha _{2}}(t_{1},t_{2}),\mathrm{Var}%
N_{\alpha _{1},\alpha _{2}}(t_{1},t_{2})$ and $\mathrm{Cov}(N_{\alpha
_{1},\alpha _{2}}(t_{1},t_{2}),N_{\alpha _{1},\alpha _{2}}(s_{1},s_{2}))$
are given by (\ref{EXP1}), (\ref{VAR}), (\ref{COVF}), respectively.
\end{enumerate}
\end{proposition}

The proof is given in \cite{LRMT}, see also Appendix for more details
and more general results hold for any L\'{e}vy random fields.

%\begin{remark}
%Note that formulae (\ref{EXP1}),(\ref{VAR}) and (\ref{COVF}) are corrected
%versions of the formulae (26) and (27) of the paper \cite{LM}.
%\end{remark}

\begin{remark}
Following the ideas of this paper, the Hurst index of the Fractional Poisson
random field in $d=2$ can be defined as follows: 
\begin{equation*}
H=\inf \left\{ \beta :\lim \sup_{T\rightarrow \infty }\frac{\mathrm{Var}%
N_{\alpha _{1},\alpha _{2}}(T,T)}{{T^{2d\beta}}}<\infty \right\} %\in 
= \frac{\alpha _{1}+\alpha _{2}}{2} \in(0,1).
\end{equation*}
\end{remark}

\begin{remark}
Any random field 
\begin{equation*}
Z(t_{1},t_{2})=N(Y_{1}(t_{1}),Y_{2}(t_{2})),\quad (t_{1},t_{2})\in \mathbb{R}%
_{+}^{2}
\end{equation*}
defined on the positive quadrant $\mathbb{R}_+^2$ can be extended in the
whole space $\mathbb{R}^{2}$ in the following way: let $%
Z_{j}(t_{1},t_{2}),(t_{1},t_{2})\in \mathbb{R}_{+}^{2},\ j=1,2,3,4$ be
independent copies of the random field $Z(t_{1},t_{2}),(t_{1},t_{2})\in 
\mathbb{R}_{+}^{2}.$

Then one can define 
\begin{equation*}
\bar{Z}(t_{1},t_{2})=
\begin{cases}
\phantom{-}Z_{1}(t_{1},t_{2}), & t_{1}\geq 0,  t_{2}\geq 0 \\ 
-Z_{2}(-t_{1}^-,t_{2}),\  & t_{1}<0, t_{2}\geq 0 \\ 
-Z_{3}(t_{1},-t_{2}^-),\  & t_{1}\geq 0, t_{2}<0 \\ 
\phantom{-}Z_{4}(-t_{1}^-,-t_{2}^-), & t_{1}<0, t_{2}<0%
\end{cases}
\end{equation*}%
Therefore, modifying the cadlag property we obtain a Poisson like random
field 
$\bar{Z}(t_{1},t_{2}),(t_{1},t_{2})\in \mathbb{R}^{2}$ which has a similar
covariance structure (replacing $t_{1},t_{2},s_{1},s_{2}$ by $%
|t_{1}|,|t_{2}|,|s_{1}|,|s_{2}|$).
\end{remark}

\subsection{Characterization on increasing paths}

Let $L_{\alpha }=\{L_{\alpha }(t),t\geq 0\}$, be an $\alpha$-stable
subordinator, and $Y_{\alpha }=\{Y_{\alpha }(t),t\geq 0\} $ be its inverse ($%
\alpha \in (0,1)$). Recall that $L_{\alpha }(t)$ is a cadlag strictly
increasing process, while $Y_{\alpha }$ $(t)$ is nondecreasing and
continuous. As a consequence, the latter defines a random nonnegative
measure $\mu _{\alpha }$ on $({\mathbb{R}}_{+},{\mathcal{B}}_{{\mathbb{R}}%
_{+}})$ such that $\mu _{\alpha }([0,t])=Y_{\alpha }(t)$. The $\sigma $%
-algebra $\mathcal{G}$ contains all the information given by $\mu _{\alpha }$%
: 
\begin{equation*}
\mathcal{G}:=\sigma (L_{\alpha }(t),t\geq 0)=\sigma (Y_{\alpha }(t),t\geq
0)=\sigma (\mu _{\alpha }(B),B\in {\mathcal{B}}_{{\mathbb{R}}_{+}}).
\end{equation*}
Now, let $X(t) = N(Y_{\alpha}(t))$ be a FPP, where $N$ has intensity $\lambda$. 
We denote by $\{\mathcal{F}_t^X,t\in\mathbb{R}_+\}$
its natural filtration. 
%and by $Z(t)$ the
% inverse of $Y_{\alpha }:$ 
%\begin{equation*}
%Z(t)=\inf \{v:\,Y_{\alpha }(v)\geq t\};
%\end{equation*}%
%then each $Z(t)$ 
We note that each $\mu _{\alpha }([0,t])$ is $\mathcal{G}$-measurable, 
while $N(w)-N(\mu _{\alpha }([0,s]))$ is independent of $\sigma(\mathcal{F}_s^X,\mathcal{G}) $
for any 
% $B_s\subseteq [0,s]$ and 
$w\geq \mu _{\alpha }([0,s])$. 
Hence, for any bounded $\mathcal{F}_s^X$-measurable random variable $Y(s)$, we have
\begin{align*}
{\rm E}\Big( \int_0^\infty Y(s) {\indI}_{(s,t]}(v)\, dX_v\Big) 
& = 
{\rm E}\Big( Y(s) {\rm E}\Big( \int_0^\infty {\indI}_{(s,t]} (v)\, N(\mu _{\alpha }(dv))\Big| \sigma(\mathcal{F}_s^X,\mathcal{G}) \Big) \Big) 
\\
& = 
{\rm E}\Big( Y(s) \int_0^\infty {\indI}_{(\mu _{\alpha }([0,s]),\mu _{\alpha }([0,t])]} (w)\, 
{\rm E} ( dN_w | \sigma(\mathcal{F}_s^X,\mathcal{G}) ) \Big) 
\\
& = 
{\rm E}\Big( Y(s) \lambda \mu _{\alpha }((s,t]) \Big) 
\\
& = {\rm E}\Big( \int_0^\infty Y(s) {\indI}_{(s,t]} (v)\,\lambda \mu _{\alpha }(dv) \Big) .
\end{align*}
In other words, by \cite[Theorem T4]{BR}, the FPP $X$ is a doubly stochastic 
Poisson process with respect to the filtration $\{\sigma(\mathcal{F}_t^X,\mathcal{G}), t\in\mathbb{R}_+\} $.
Therefore a first characterization of a FPP may be written in the following way.
\begin{corollary}
A process $N_{\alpha}$ 
is a FPP iff it is a doubly stochastic 
Poisson process 
with intensity $\lambda Y_{\alpha}$, with respect to the filtration $\{\sigma(\mathcal{F}_t^X,\mathcal{G}), t\in\mathbb{R}_+\} $.
In other words,
whenever $B_1,\ldots, B_n$ are disjoint bounded Borel sets and $x_1, \ldots,
x_n$ are non-negative integers, then 
\begin{equation*}%\label{eq:FPPinPoiss}
{\rm P}\Big( \bigcap_{i=1}^n \{ N_{\alpha}(B_i)=x_i \} \Big| \mathcal{G }\Big) =
\prod_{i=1}^n \frac{ \exp(-\lambda \mu_{\alpha}(B_i)) (\lambda
\mu_{\alpha}(B_i))^{x_i} }{ x_i!}.
\end{equation*}
\end{corollary}

An analogous result may be found for FPRF. In fact,
let $Y_{\alpha _{1}}^{(1)}(t),t\geq 0$ and $Y_{\alpha _{2}}^{(2)}(t),t\geq 0$
be two independent inverse stable subordinators with indices $\alpha _{1}\in
(0,1)$ and $\alpha _{2}\in (0,1)$. Let $\mu_{\alpha_1}$ and $\mu_{\alpha_2}$%
, $\mathcal{G}_1$ and $\mathcal{G}_2$ their respective $\sigma$-algebras (this
notation will be used in the following results).

If $\mu_{\boldsymbol{\alpha}} = \mu_{\alpha_1}\otimes\mu_{\alpha_2}$ is the
product measure and $\mathcal{G }= \sigma(\mathcal{G}_1,\mathcal{G}_2)$, we
can follow the same reasoning as above once we have noted that 
$ \Delta_{\mu_{\alpha_1}([0,s_{1}]),\mu_{\alpha_2}([0,s_{2}])} N (w_{1},w_{2}) $
and $\sigma(\mathcal{F}^X_{\infty,s_2} \vee \mathcal{F}^X_{s_1, \infty} ) $
are conditionally independent, 
given $\mathcal{G}$.
In fact
\begin{multline}\label{eq:str_mtgale}
{\rm E}\Big( \Delta_{s_{1},s_{2}} X (t_{1},t_{2}) 
% - \Delta_{s_{1},s_{2}} \mu_{\boldsymbol{\alpha}}(t_{1},t_{2})  
\Big| 
\sigma(\mathcal{F}^X_{\infty,s_2} \vee \mathcal{F}^X_{s_1, \infty} ,\mathcal{G} )
\Big) \\
\begin{aligned}
& = 
{\rm E}\Big( \Delta_{\mu_{\alpha_1}([0,s_{1}]),\mu_{\alpha_2}([0,s_{2}])} N (\mu_{\alpha_1}([0,t_{1}]),\mu_{\alpha_2}([0,t_{2}]))
% - \Delta_{s_{1},s_{2}} \mu_{\boldsymbol{\alpha}}(t_{1},t_{2})  
\Big| 
\sigma(\mathcal{F}^X_{\infty,s_2} \vee \mathcal{F}^X_{s_1, \infty} ,\mathcal{G})
\Big) 
\\
& = 
{\rm E}\Big( \Delta_{\mu_{\alpha_1}([0,s_{1}]),\mu_{\alpha_2}([0,s_{2}])} N (\mu_{\alpha_1}([0,t_{1}]),\mu_{\alpha_2}([0,t_{2}]))
\Big| 
\mathcal{G} 
\Big) 
\\
& = \lambda \mu_{\boldsymbol{\alpha}} ( ( (s_1,s_2),(t_1,t_2) ]) .
\end{aligned}
\end{multline}
In other words, 
the FPRF $X$ is a $\mathcal{F}^*$-doubly stochastic 
Poisson process (see \cite{MN} for the definition of $\mathcal{F}^*$-doubly stochastic 
Poisson process) 
with respect to the filtration 
$\{\sigma(\mathcal{F}^X_{t_1,t_2} ,\mathcal{G})) , (t_1,t_2)\in\mathbb{R}^2_+\} $ by \cite[Theorem~1]{MN}. Again,
we may summarize this result in the following statement.

\begin{proposition}\label{thm:char1}
A process $N_{\alpha _{1},\alpha _{2}}$ is a FPRF iff 
it is a $\mathcal{F}^*$-doubly stochastic 
Poisson process 
with intensity $\lambda Y_{\alpha_1}\cdot Y_{\alpha_2}$, with respect to the filtration 
$\{\sigma(\mathcal{F}^X_{t_1,t_2} ,\mathcal{G})) , t_1,t_2\in\mathbb{R}_+\} $.
In other words, whenever $B_1,\ldots, B_n$ are
disjoint bounded Borel sets in ${\mathbb{R}}_+\times {\mathbb{R}}_+$ and $%
x_1, \ldots, x_n$ are non-negative integers, then 
\begin{equation}\label{eq:FPRFinPoiss}
{\rm P}\Big( \bigcap_{i=1}^n \{ N_{\alpha _{1},\alpha _{2}}(B_i)=x_i \} \Big| 
\mathcal{G }\Big) = \prod_{i=1}^n \frac{ \exp(-\lambda \mu_{\boldsymbol{%
\alpha}}(B_i)) (\lambda \mu_{\boldsymbol{\alpha}}(B_i))^{x_i} }{ x_i!}.
\end{equation}
\end{proposition}
Now, let $t_1>0$ be fixed. The process $t\mapsto N_{\alpha _{1},\alpha _{2}}(t_1,t)$ is the trace of the FPRF
along the increasing $t$-indexed family of sets $t \mapsto [(0,0),(t_1,t)]$.
As a consequence of the previous results, we obtain:
\begin{theorem}\label{thm:along1and2}
A random simple locally finite counting measure $N_{\alpha _{1},\alpha _{2}}$ is a FPRF iff 
$\mathcal{G}_1,\mathcal{G}_2$ are independent, and
fixed $t_1,t_2\geq 0$,
the process $N_{\alpha _{1},\alpha _{2}}(t_1,t)$, conditioned on $\mathcal{%
G}_1$, is a FPP $N_{\alpha _{2}}(t)$, the process $N_{\alpha _{1},\alpha _{2}}(t,t_2)$, 
conditioned on $\mathcal{%
G}_2$, is a FPP $N_{\alpha _{1}}(t)$, and the two processes 
$N_{\alpha _{1}}(t_1+t)-N_{\alpha _{1}}(t_1), N_{\alpha _{2}}(t_2+t)-N_{\alpha _{2}}(t_2)$ are
conditionally independent given $\sigma(\mathcal{G},\sigma(N_{\alpha _{1},\alpha _{2}}(s_1,s_2),
(s_1,s_2)\prec (t_1,t_2)))$.
\end{theorem}
\begin{proof}
Assume that $N_{\alpha _{1},\alpha _{2}}$ is a FPRF and $t_1>0$ fixed.
Denote by $X_t = N_{\alpha _{1},\alpha _{2}}(t_1,t)$ and note that 
$\sigma( \{Y_{\alpha_{2}}(t),t\geq 0\}) = \mathcal{G}_2$. Let
$B_1,\ldots, B_n$ be disjoint bounded Borel sets and $x_1, \ldots,
x_n$ non-negative integers. We have
\begin{multline*}
{\rm P}\Big( \bigcap_{i=1}^n \{ N_{\alpha _{1},\alpha _{2}}([0,t_1] \times B_i)=x_i \} \Big| 
\sigma(\mathcal{G}_1,\sigma( \{Y_{\alpha_{2}}(t),t\geq 0\}))\Big) \\
\begin{aligned}
& = 
{\rm P}\Big( 
\bigcap_{i=1}^n \{ N_{\alpha _{1},\alpha _{2}}([0,t_1] \times B_i)=x_i \} \Big| 
\mathcal{G} \Big) \\
& = 
\prod_{i=1}^n \frac{ \exp(-\lambda \mu_{\boldsymbol{%
\alpha}}([0,t_1] \times B_i)) (\lambda \mu_{\boldsymbol{\alpha}}([0,t_1] \times B_i))^{x_i} }{ x_i!}
\\
& = 
\prod_{i=1}^n \frac{ \exp(-\lambda 
Y_{\alpha_{1}}(t_1) \cdot \mu_{\alpha_2}( B_i) 
) (\lambda Y_{\alpha_{1}}(t_1) \cdot \mu_{\alpha_2}( B_i) )^{x_i} }{ x_i!},
\end{aligned}
\end{multline*}
and hence $X_t= M(Y_{\alpha_2}(t))$, where, conditioned on $\mathcal{%
G}_1$, $M$ is a Poisson process with intensity $\lambda 
Y_{\alpha_{1}}(t_1) $. 
The conditional independence follows by similar arguments, and hence the first implication is proved.
%The process $M$ is also independent on $ \mathcal{G}_1$, while 
%its intensity is deterministic conditioned to $ \mathcal{G}_1$. 

Conversely,  by \cite{MN}, to prove Proposition~\ref{thm:char1} it is sufficient to prove
\eqref{eq:str_mtgale}. 
Denote by 
\begin{align*}
\mathcal{H}^1_{s_1,s_2} & = \sigma( {N_{\alpha_1,\alpha_2}}(s_1+t,s) -{N_{\alpha_1,\alpha_2}}(s_1,s), t\geq 0,s\leq s_2 )
\\
\mathcal{H}^2_{s_1,s_2} &= \sigma( {N_{\alpha_1,\alpha_2}}(s,s_2+t) -{N_{\alpha_1,\alpha_2}}(s,s_2) , t\geq 0,s\leq s_1),
\end{align*}
so that 
$\mathcal{F}^{N_{\alpha_1,\alpha_2}}_{\infty,s_2} = \sigma(\mathcal{F}^{N_{\alpha_1,\alpha_2}}_{s_1,s_2},\mathcal{H}^1_{s_1,s_2}  )$ and
$\mathcal{F}^{N_{\alpha_1,\alpha_2}}_{s_1,\infty} = \sigma(\mathcal{F}^{N_{\alpha_1,\alpha_2}}_{s_1,s_2},\mathcal{H}^2_{s_1,s_2} )$.
Then, denoting by $X \ci Y | W$ the conditional independence of $X$ and $Y$, given $W$, we have
by hypothesis that 
\[
\mathcal{H}^1_{s_1,s_2} \ci \mathcal{H}^2_{s_1,s_2} | 
\sigma(\mathcal{G}, \mathcal{F}^{N_{\alpha_1,\alpha_2}}_{s_1,s_2}), \qquad
\mathcal{H}^1_{s_1,s_2} \ci \mathcal{F}^{N_{\alpha_1,\alpha_2}}_{s_1,s_2} | 
\mathcal{G}, \qquad
\mathcal{H}^2_{s_1,s_2} \ci \mathcal{F}^{N_{\alpha_1,\alpha_2}}_{s_1,s_2} | 
\mathcal{G}, 
\] 
for any $(s_1,s_2)$. 
Thus, 
\begin{itemize}
\item 
$ \mathcal{H}^2_{t_1,t_2} \ci \mathcal{F}^{N_{\alpha_1,\alpha_2}}_{t_1,t_2}, \mathcal{H}^1_{t_1,t_2} | 
\mathcal{G}, 
$,
$ \mathcal{F}^{N_{\alpha_1,\alpha_2}}_{t_1,t_2} \ci \mathcal{H}^1_{t_1,t_2} | 
\mathcal{G}, 
$, $\mathcal{H}^1_{t_1,s_2}\subseteq \mathcal{H}^1_{t_1,t_2}$, 
$\mathcal{H}^2_{s_1,t_2}\subseteq \mathcal{H}^2_{t_1,t_2}$, then
\[
{\rm E}\Big( 
{N_{\alpha_1,\alpha_2}}(t_{1},t_{2})
\Big| 
\sigma(\mathcal{F}^{N_{\alpha_1,\alpha_2}}_{\infty,s_2} \vee \mathcal{F}^{N_{\alpha_1,\alpha_2}}_{s_1, \infty} ,\mathcal{G})
\Big) = 
{\rm E}\Big( 
{N_{\alpha_1,\alpha_2}}(t_{1},t_{2})
\Big| 
\sigma(\mathcal{F}^{N_{\alpha_1,\alpha_2}}_{t_1,s_2} \vee \mathcal{F}^{N_{\alpha_1,\alpha_2}}_{s_1, t_2} ,\mathcal{G})
\Big) ,
\]
and hence
\begin{equation}\label{eq:n1}
{\rm E}\Big( \Delta_{s_{1},s_{2}} {N_{\alpha_1,\alpha_2}} (t_{1},t_{2}) 
% - \Delta_{s_{1},s_{2}} \mu_{\boldsymbol{\alpha}}(t_{1},t_{2})  
\Big| 
\sigma(\mathcal{F}^{N_{\alpha_1,\alpha_2}}_{\infty,s_2} \vee \mathcal{F}^{N_{\alpha_1,\alpha_2}}_{s_1, \infty} ,\mathcal{G})
\Big) 
=
{\rm E}\Big( \Delta_{s_{1},s_{2}} {N_{\alpha_1,\alpha_2}} (t_{1},t_{2}) 
% - \Delta_{s_{1},s_{2}} \mu_{\boldsymbol{\alpha}}(t_{1},t_{2})  
\Big| 
\sigma(\mathcal{F}^{N_{\alpha_1,\alpha_2}}_{t_1,s_2} \vee \mathcal{F}^{N_{\alpha_1,\alpha_2}}_{s_1, t_2} ,\mathcal{G})
\Big) ;
\end{equation}

\item note that $\mathcal{F}^{N_{\alpha_1,\alpha_2}}_{t_1,s_2} = \sigma(\mathcal{F}^{N_{\alpha_1,\alpha_2}}_{s_1,s_2} ,\mathcal{H})$,
where $\mathcal{H} = \sigma( \Delta_{s_{1},s_{2}} {N_{\alpha_1,\alpha_2}} (u,v), s_1\leq u \leq t_1,v \leq s_2 )$. In addition,
$\mathcal{H}^1_{s_1,t_2} \ci \mathcal{F}^{N_{\alpha_1,\alpha_2}}_{s_1,t_2} | 
\mathcal{G} $, and $\sigma ( \Delta_{s_{1},s_{2}} {N_{\alpha_1,\alpha_2}} (t_{1},t_{2}) , \mathcal H ) \subseteq \mathcal{H}^1_{s_1,t_2} $.
Hence
\begin{equation}\label{eq:n2}
{\rm E}\Big( \Delta_{s_{1},s_{2}} {N_{\alpha_1,\alpha_2}} (t_{1},t_{2}) 
% - \Delta_{s_{1},s_{2}} \mu_{\boldsymbol{\alpha}}(t_{1},t_{2})  
\Big| 
\sigma(\mathcal{F}^{N_{\alpha_1,\alpha_2}}_{t_1,s_2} \vee \mathcal{F}^{N_{\alpha_1,\alpha_2}}_{s_1, t_2} ,\mathcal{G})
\Big) = 
{\rm E}\Big( \Delta_{s_{1},s_{2}} {N_{\alpha_1,\alpha_2}} (t_{1},t_{2}) 
% - \Delta_{s_{1},s_{2}} \mu_{\boldsymbol{\alpha}}(t_{1},t_{2})  
\Big| 
\sigma(\mathcal{H}  ,\mathcal{G})
\Big) ;
\end{equation}

\item now, note that both ${N_{\alpha_1,\alpha_2}} (t_{1},t_{2}) - 
{N_{\alpha_1,\alpha_2}}(t_{1},s_{2}) $ and ${N_{\alpha_1,\alpha_2}} (s_{1},t_{2}) - 
{N_{\alpha_1,\alpha_2}}(s_{1},s_{2}) $ belong to $\mathcal{H}^2_{t_1,s_2}$, while $\mathcal{H} \subseteq \mathcal{F}^{N_{\alpha_1,\alpha_2}}_{t_1, s_2} $. Hence
\begin{align}\label{eq:n3}
{\rm E}\Big( \Delta_{s_{1},s_{2}} {N_{\alpha_1,\alpha_2}} (t_{1},t_{2}) 
% - \Delta_{s_{1},s_{2}} \mu_{\boldsymbol{\alpha}}(t_{1},t_{2})  
\Big| 
\sigma(\mathcal{H}  ,\mathcal{G})
\Big) = {\rm E}( \Delta_{s_{1},s_{2}} {N_{\alpha_1,\alpha_2}} (t_{1},t_{2}) 
% - \Delta_{s_{1},s_{2}} \mu_{\boldsymbol{\alpha}}(t_{1},t_{2})  
| 
\mathcal{G}) .
\end{align}
\end{itemize}
Combining \eqref{eq:n1}, \eqref{eq:n2} and \eqref{eq:n3} we finally get \eqref{eq:str_mtgale}:
\begin{multline*}
{\rm E}\Big( \Delta_{s_{1},s_{2}} {N_{\alpha_1,\alpha_2}} (t_{1},t_{2}) 
% - \Delta_{s_{1},s_{2}} \mu_{\boldsymbol{\alpha}}(t_{1},t_{2})  
\Big| 
\sigma(\mathcal{F}^{N_{\alpha_1,\alpha_2}}_{\infty,s_2} \vee \mathcal{F}^{N_{\alpha_1,\alpha_2}}_{s_1, \infty} ,\mathcal{G})
\Big) 
= {\rm E}( \Delta_{s_{1},s_{2}} {N_{\alpha_1,\alpha_2}} (t_{1},t_{2}) 
% - \Delta_{s_{1},s_{2}} \mu_{\boldsymbol{\alpha}}(t_{1},t_{2})  
| 
\mathcal{G}) \\
= \lambda (Y_{\alpha_1}(t_1)-Y_{\alpha_1}(s_1)) (Y_{\alpha_2}(t_2)-Y_{\alpha_2}(s_2)). %\qedhere
\end{multline*}
\end{proof}

Let $\mathcal{A }$ be the collection of the closed rectangles $%
\{A_{t_1,t_2} : t\in \mathbb{R}_{+}^2 \}$, where $A_{t_1,t_2} = \{ (s_1,s_2) \in \mathbb{R}_{+}^2 :
0 \leq s_i \leq t_i, i=1,2\}$. The family $\mathcal{A }$ generates a topology
of closed sets $\widetilde{\mathcal{A}}(u)$, which is closed under finite
unions and arbitrary intersections, called a \emph{lower set} family (see,
e.g., \cite{A,IM1}). In other words, when a point $(t_1,t_2)$ belongs to a set $A \in \widetilde{\mathcal{A}}(u)$,
all the rectangle $A_{t_1,t_2}$ is contained in $ A$:
\[
A \in \widetilde{\mathcal{A}}(u) \qquad \iff \qquad  A_{t_1,t_2} \subseteq A , \forall (t_1,t_2)\in A .
\]

A function $\Gamma :\mathbb{R}_{+}\rightarrow \widetilde{%
\mathcal{A}}(u) $ is called an \emph{increasing set} if $\Gamma(0) = \{(0,0)\}$, it is continuous,
%($\Gamma (t)=\lim_{s\downarrow t}\Gamma (s)$) 
it is non-decreasing ($s\leq t\Longrightarrow \Gamma (s)\subseteq
\Gamma (t)$), and the area it underlies is finite for any $t$ and goes to infinity when $t$ increases
($\lim_{t\to +\infty} | \Gamma (t)| = \infty$). 
Note that, for a
nonnegative measure $\mu $ on $B_{\mathbb{R}_{+}\times \mathbb{R}_{+}}$, it is well-defined the
non-decreasing right-continuous function:
\begin{equation*}
(\mu \circ \Gamma) (t)=\mu (\Gamma (t)).
\end{equation*}%
Accordingly, given an increasing path $\Gamma $ and a random nonnegative measure $N$ (in \cite{IM1}, it is an increasing and
additive process), we may define the
one-parameter process $N\circ \Gamma $ as the trace of $N$ along $\Gamma $: 
\begin{equation*}
(N\circ \Gamma )(t)=N(\{\Gamma (t)\}),\quad t\ge0.
\end{equation*}

Theorem~\ref{thm:along1and2} shows an example of characterizations of FPRF.
In \cite{IM0}, the authors proved a characterization of the inhomogeneous Poisson
processes on the plane thorough its realizations on increasing families of points (called increasing path) and 
increasing families of sets, called increasing set
(see also \cite{AC,IMP}). 

We are going to characterize an FPRF in the same spirit.

\begin{theorem}\label{thm:charL1}
A random simple locally finite counting measure  $N_{\alpha _{1},\alpha _{2}}$ is a FPRF iff, conditioned on $%
\mathcal{G}$, $N \circ \Gamma$ is a one-parameter inhomogeneous Poisson process with
intensity $\lambda (\mu_{\boldsymbol{\alpha}} \circ \Gamma)$, for any increasing set 
$\Gamma $, independent of $\mathcal{G}$.
\end{theorem}
\begin{proof}
Assume that $N_{\alpha _{1},\alpha _{2}}$ is a FPRF. Then, for any $0\leq s_1<t_1\leq s_2<t_2\leq\cdots \leq s_n<t_n$,
the sets $B_i = \Gamma (t_{i}) \setminus \Gamma (s_{i})$ are disjoint. By \eqref{eq:FPRFinPoiss},
\begin{align*}
{\rm P}\Big( \bigcap_{i=1}^n \{ (N \circ \Gamma ) (s_i, t_i] =x_i \} \Big| 
\mathcal{G }\Big)  
& = 
{\rm P}\Big( \bigcap_{i=1}^n \{ N_{\alpha _{1},\alpha _{2}}(B_i)=x_i \} \Big| 
\mathcal{G }\Big) 
\\
& = 
\prod_{i=1}^n \frac{ \exp(-\lambda \mu_{\boldsymbol{%
\alpha}}(B_i)) (\lambda \mu_{\boldsymbol{\alpha}}(B_i))^{x_i} }{ x_i!}\\
& = 
\prod_{i=1}^n \frac{ \exp\big(-\lambda \cdot (\mu_{\boldsymbol{\alpha}} \circ \Gamma)
(s_i, t_i] \big) \big(\lambda \cdot (\mu_{\boldsymbol{\alpha}} \circ \Gamma) (s_i, t_i] \big)^{x_i} }{ x_i!}
.
\end{align*}

Conversely, note that that \eqref{eq:FPRFinPoiss} may be checked only on disjoint rectangles $B_1, B_2, \ldots,B_n$ (see also \cite{IM1}).
After ordering partially the rectangles with respect to $\prec$, one can build an increasing sets $\Gamma$ such that
\(
B_i = \Gamma (t_{i}) \setminus \Gamma (s_{i}) ,
\)
where $0\leq s_1<t_1\leq s_2<t_2\leq\cdots \leq s_n<t_n$. By hypothesis, $N \circ \Gamma$ is
an inhomogeneous Poisson process with
intensity $\mu_{\boldsymbol{\alpha}} \circ \Gamma$. Then,
\begin{align*}
{\rm P}\Big( \bigcap_{i=1}^n \{ N_{\alpha _{1},\alpha _{2}}(B_i)=x_i \} \Big| 
\mathcal{G }\Big) 
& = 
{\rm P}\Big( \bigcap_{i=1}^n \{ (N \circ \Gamma ) (s_i, t_i] =x_i \} \Big| 
\mathcal{G }\Big)  
\\
& = 
\prod_{i=1}^n \frac{ \exp\big(-\lambda \cdot (\mu_{\boldsymbol{\alpha}} \circ \Gamma)
(s_i, t_i] \big) \big(\lambda \cdot (\mu_{\boldsymbol{\alpha}} \circ \Gamma) (s_i, t_i] \big)^{x_i} }{ x_i!}
\\
& = 
\prod_{i=1}^n \frac{ \exp(-\lambda \mu_{\boldsymbol{%
\alpha}}(B_i)) (\lambda \mu_{\boldsymbol{\alpha}}(B_i))^{x_i} }{ x_i!}. %\qedhere
\end{align*}
\end{proof}

Now, a function $\Gamma :\mathbb{R}_{+}\rightarrow \mathbb{R}_{+}^2$ is called an
\emph{increasing path} if $\Gamma(0) = (0,0)$, it is continuous,
%($\Gamma (t)=\lim_{s\downarrow t}\Gamma (s)$) 
it is non-decreasing ($s\leq t\Longrightarrow \Gamma_1 (s)\leq
\Gamma_1 (t), \Gamma_2 (s)\leq \Gamma_2 (t) $), and the area it underlies goes to infinity
($\lim_{t\to +\infty}\Gamma_1 (t)\Gamma_2 (t) = \infty$). 
In other words, an increasing path is an increasing set where, for each $t$, $\Gamma(t)$ is a rectangle.
Given an increasing path $\Gamma $ and a process $%
N(t_{1},t_{2}) $, the one-parameter process $N\circ \Gamma $ is
the trace of $N$ along $\Gamma $: 
\begin{equation*}
(N\circ \Gamma )(t)=
 \Delta_{0,0} N (\Gamma_1 (t), \Gamma_2 (t)) 
=N(\Gamma _{1}(t),\Gamma _{2}(t)), \quad t\ge0.
\end{equation*}

When dealing with the laws of the traces of a process along increasing paths, one cannot hope to prove, for instance, 
the conditional independence of two filtrations as $\mathcal{H}^1_{s_1,s_2}$ and $\mathcal{H}^2_{s_1,s_2}$, since
the event that belong to those filtrations are generated by the increments of the process on regions that are notù
comparable with respect to the partial order $\prec$.

As an example, there is no increasing path that separates the three rectangles $B_1=\{(1,0)\prec z \prec (2,1)\}$,
$B_2=\{(0,1)\prec z \prec (1,2)\}$ and $B_3=\{(1,1)\prec z \prec (2,2)\}$ and hence we cannot give the joint 
law of $\Delta_{(1,0)} N (2,1)$ and $\Delta_{(0,1)} N (1,2)$. On the other hand, Proposition~\ref{thm:char1} 
suggests that, if we assume the independence of $N(B_1)$ and $N(B_2)$ conditioned on $\mathcal F_{1,1}$, the equation
\eqref{eq:FPRFinPoiss} may be proved for $B_1$, $B_2$ and $B_3$ via increasing paths (as in \cite{AC,AC2,IM0,IMP}).
This consideration has suggested the following definition.

We say that the filtration satisfies the conditional independence condition
or the Cairoli-Walsh condition ((F4) in \cite{CW}, see also \cite{KS81}) if for any $\mathcal{F}$%
-measurable integrable random variable $Z,$ and for any $(t_{1},t_{2}):$%
\begin{equation*}
\mathrm{E} ( \mathrm{E} ( Z \vert \mathcal{F}_{t_{1},\infty}) \vert \mathcal{%
F}_{\infty,t_{2}})=\mathrm{E} ( \mathrm{E} ( Z \vert \mathcal{F}%
_{\infty,t_{2}}) \vert \mathcal{F}_{t_{1},\infty})=\mathrm{E} ( Z\vert 
\mathcal{F}_{t_{1},t_{2}}).
\end{equation*}
Thus, following the same ideas as in \cite{AC,AC2,IM0,IMP}, one can prove the following result.

\begin{theorem}\label{thm:charL2}
A random simple locally finite counting measure $N_{\alpha _{1},\alpha _{2}}$ is a FPRF iff, conditioned on $%
\mathcal{G}$, the Cairoli-Walsh condition holds and $N\circ \Gamma $ is an
inhomogeneous Poisson process with intensity $Y_{\alpha _{1}}(\Gamma
_{1}(t)) \cdot Y_{\alpha _{2}}(\Gamma _{2}(t))$, for any increasing path $\Gamma %
%:\mathbb{R}_{+}\rightarrow \mathbb{R}_{+}^2
$.
\end{theorem}

%\begin{remark}
\subsubsection*{A remark on Set-Indexed Fractional Poisson Process}
Let $T$ be a metric space equipped with a Radon measure on its Borel sets. We assume existence of an indexing collection $\mathcal{A}$ on $T$, as it is defined in \cite{IM1}. We are interested to considering processes indexed by a class of closed sets from $T$.
In this new framework, 
$\Gamma:\mathbb{R}_+\to\mathcal{A}$ is called an increasing path if it is
continuous and increasing: $s<t \Longrightarrow \Gamma(s)\subseteq \Gamma(t)$ (called a flow
in \cite{HM1})

\noindent We can now define Set-Indexed Fractional Poisson process.

A set-indexed process $X=\{X_{U},\,U\in \mathcal{A}\}$ is called a
Set-Indexed Fractional Poisson process(SIFPP), if for any increasing path $%
\Gamma $ the process $X^{\Gamma }=\{X_{\Gamma (t)},\,t\geq 0\}$ is an FPP.

\begin{remark}
Following results of \cite{IM1}, we can state that
%\begin{theorem}
any SIFPP is a set-indexed L\'evy process.
\end{remark}
\noindent Details and martingale characterizations will be presented elsewhere.
%\end{remark}

\subsection{Gergely-Yezhow characterization}
Let $(U_n,n\geq 1)$ be a sequence of i.i.d.\ $(0,1)$-uniform distributed random variables, independent of
the processes $Y_{\alpha _{i}}$, $i=1,2$.
The random indexes associated to the `records' $(\nu_n, n\geq 1)$ are inductively defined by 
\[
\nu_1 (\omega) = 1, \qquad \nu_{n+1} (\omega) = \inf \{ k > \nu_n(\omega) \colon U_k(\omega) > 
U_{\nu_n(\omega)} (\omega) \}.
\]
It is well known (see, e.g., \cite[p.63-78]{And01}) that $P(\cap_n \{\nu_n < \infty\} ) =1$, and hence the $k$-th record 
$V_k$ of the sequence is well defined: $V_0 := 0$, $V_k = U_{\nu_k}$. Since $V_n \geq \max(U_1,\ldots,U_n)$, then $P(V_n \to 1) = 1$. 
Moreover, the number of $U_n$'s that realize 
the maximum by time $n$ is almost surely asymptotic to $\log (n) $ as $n\to\infty$. In other words,
the sequence $(\nu_n)_n$ growths exponentially fast.

Now, given a increasing set $\Gamma$, we define
\[
Y_t^{\Gamma} = \sum_n n {\indI}_{[V_{n},V_{n+1})} (1- \exp( - \mu_{\boldsymbol{\alpha}} \circ \Gamma (t) ))
= \sup\{ n \colon V_n \leq 1- \exp( - \mu_{\boldsymbol{\alpha}} \circ \Gamma (t) )\}.
\]

\begin{theorem}
A random simple locally finite counting measure $N_{\alpha _{1},\alpha _{2}}$ is a FPRF iff $N \circ \Gamma$ 
is distributed as $Y^{\Gamma} $, for any increasing set 
$\Gamma$.%\label{thm:GY-c}
\end{theorem}
\begin{proof}
In the proof %of the following Theorem~\ref{thm:GY-c}, 
we assume that $ \lim_t \mu_{\boldsymbol{\alpha}} \circ \Gamma (t) = \infty$ almost surely.
When this is not the case, the proof should be changed as in \cite{GY}, where generalized random variables are introduced exactly
when $1- \exp(-\text{``intensity at $\infty$''}) < 1$.

By Theorem~\ref{thm:charL1}, we must prove that, conditioned on $%
\mathcal{G}$, $Y^\Gamma$ is an inhomogeneous Poisson process with
intensity $\mu_{\boldsymbol{\alpha}} \circ \Gamma$.
Conditioned on $\mathcal G$, let $F(t):= 1- \exp( - \mu_{\boldsymbol{\alpha}} \circ \Gamma (t) ) $
be the continuous deterministic cumulative distribution function. Let $F^{-}$ be its pseudo-inverse
$
F^{-} (x) = \inf\{ y \colon F(y) > x\},   
$
and
define $\xi_n=F^{-}(U_n)$, for each $n$.
Then $(\xi_n,n\geq 1)$ is a sequence of i.i.d.\ random variables with cumulative function $F$.
As in \cite{GY}, put $\zeta'_n= \max(\xi_1,\ldots,\xi_n)$, ($n = 1,2,\ldots$) omitting in the increasing sequence
\[
\zeta'_1,\zeta'_2,\ldots, \zeta'_n, \ldots
\]
all the repeating elements except one, we come to the strictly increasing sequence \cite[Eq.~(3)]{GY}
\[
\zeta_1,\zeta_2,\ldots, \zeta_n, \ldots
\]
Now, since $F^{-}$ is monotone, it is obvious by definition that $\zeta_n = F^{-} (V_n)$.
Again, $F^{-}$ is monotone, and hence
\begin{align*}
Y_t^{\Gamma} &= \sum_n n {\indI}_{[F^{-} (V_{n}),F^{-} (V_{n+1}))} (F^{-} (1- \exp( - \mu_{\boldsymbol{\alpha}} \circ \Gamma (t) )))\\
&= \sum_n n {\indI}_{[\zeta_n,\zeta_{n+1})} (t),%\\
%= \sup\{ n \colon V_n \leq 1- \exp( - \mu_{\boldsymbol{\alpha}} \circ \Gamma (t) )\}.
\end{align*}
that is the process $v(t)$ defined in \cite[Eq.~(7')]{GY}. The thesis is now an application of \cite[Theorem 1]{GY} and
Theorem~\ref{thm:charL1}.
\end{proof}

\subsection{Random time change}

The process $\mu_{\boldsymbol{\alpha}} $ may be used to reparametrize the
time of the increasing paths and sets. In fact, for any increasing path $%
\Gamma= (\Gamma_1(t),\Gamma_2(t))$, let 
\[
T(s,\omega) =  
\begin{cases}
\inf\{ t \colon
Y_{\alpha_1}(\Gamma_1(t)) \cdot Y_{\alpha_2} (\Gamma_2(t)) (\omega)> s\}
& \text{if }\{ t \colon
Y_{\alpha_1}(\Gamma_1(t)) \cdot Y_{\alpha_2} (\Gamma_2(t)) (\omega)> s\}\neq \varnothing;
\\
\infty & \text{otherwise};
\end{cases}
\]
be the first time that the intensity is seen to be bigger than $s$ on the increasing path, and define
\begin{equation}  \label{eq:reparam}
\Gamma_{\mu_{\boldsymbol{\alpha}}} (s,\omega) = \Gamma ( T(s,\omega) )
\end{equation}
the reparametrization of $\Gamma$ made by $\mu_{\boldsymbol{\alpha}} $.
Analogously,  for any increasing set $\Gamma$, let 
\begin{equation*}
\Gamma_{\mu_{\boldsymbol{\alpha}}} (s,\omega) = \Gamma ( \inf\{ t \colon
(\mu_{\boldsymbol{\alpha}}(\omega) \circ \Gamma ) (t) > s\} ).
\end{equation*}
We note that, for any fixed $s$ and $A \in \widetilde{\mathcal{A}}(u)$ 
\begin{equation}  \label{eq:optional}
\{ \omega\colon A \nsubseteq \Gamma_{\mu_{\boldsymbol{\alpha}}} (s)\} =
\cup_{t\in \mathbb{Q}} \Big(\{ A \nsubseteq \Gamma (t)\} \cap \{ {\mu_{%
\boldsymbol{\alpha}}} (\Gamma(t)\cap A) \geq s\} \Big) \in \mathcal{G}_{A} ,
\end{equation}
where $\mathcal{G}_{A} = \sigma( {\mu_{\boldsymbol{\alpha}}} (A^{\prime}),
A^{\prime}\subseteq A). $ We recall that a random measurable set $Z: \Omega
\to \widetilde{\mathcal{A}}(u)$ is called a $\mathcal{G}_{A}$-stopping set
if $\{ A \subseteq Z\} \in \mathcal{G}_{A} $ for any $A$. As a consequence,
the reparametrization given in \eqref{eq:reparam} transforms $\Gamma(\cdot)$
into $\Gamma_{\mu_{\boldsymbol{\alpha}}} (\cdot)$, a family of
continuous increasing stopping set by \eqref{eq:optional}. Such a
family is called an optional increasing set.
The random time change theorem 
(which can be made an easy consequence of the characterization of the Poisson process given in
\cite{WAT}) together with Theorem~\ref{thm:charL1} and Theorem~\ref{thm:charL2} give
the following corollaries, that can be seen as extensions of some results in \cite{AC,AC2}.

\begin{corollary}
A random simple locally finite counting measure $N_{\alpha _{1},\alpha _{2}}$ is a FPRF iff, conditioned on $%
\mathcal{G}$, $N \circ \Gamma_{\mu_{\boldsymbol{\alpha}}}$ is a standard
Poisson process, for any increasing set $\Gamma $.
\end{corollary}

\begin{corollary}
A random simple locally finite counting measure  $N_{\alpha _{1},\alpha _{2}}$ is a FPRF iff, conditioned on $%
\mathcal{G}$, the Cairoli-Walsh condition holds \cite{CW,KS81} and $N\circ \Gamma _{\mu _{\boldsymbol{\alpha }}}$
is a standard Poisson process, for any increasing path $\Gamma $.
\end{corollary}

\section{Fractional Differential Equations}\label{sect:5}

A direct calculation may be applied to show that the 
marginal distribution of the classical Poisson random field $%
N(t_{1},t_{2}),\ (t_{1},t_{2})\in \mathbb{R}_{+}^{2}$%
\begin{equation*}
p_{k}^{c}(t_{1},t_{2})=\mathrm{P}\left( N(t_{1},t_{2})=k\right) =\frac{%
e^{-\lambda t_{1}t_{2}}(\lambda t_{1}t_{2})^{k}}{k!},k=0,1,2\ldots
\end{equation*}%
satisfy the following differential-difference equations:%
\begin{align}
&\frac{\partial ^{2}p_{0}^{c}\left( t_{1},t_{2}\right) }{\partial
t_{1}~\partial t_{2}}=\left( -\lambda +\lambda ^{2}t_{1}t_{2}\right)
p_{0}^{c}\left( t_{1},t_{2}\right) ;  \label{5.1}
\\
&\frac{\partial ^{2}p_{1}^{c}\left( t_{1},t_{2}\right) }{\partial
t_{1}~\partial t_{2}}=\left( -3\lambda +\lambda ^{2}t_{1}t_{2}\right)
p_{1}^{c}\left( t_{1},t_{2}\right) +\lambda p_{0}^{c}\left(
t_{1},t_{2}\right) ;  \label{5.2}
\\
&\frac{\partial ^{2}p_{k}^{c}\left( t_{1},t_{2}\right) }{\partial
t_{1}~\partial t_{2}}=\left( -\lambda +\lambda ^{2}t_{1}t_{2}\right)
p_{k}^{c}\left( t_{1},t_{2}\right)  +\left( \lambda -2\lambda
^{2}t_{1}t_{2}\right) p_{k-1}^{c}\left( t_{1},t_{2}\right) +\lambda
^{2}p_{k-2}^{c}\left( t_{1},t_{2}\right) ;\ k\geq 2;  \label{5.3}
\end{align}%
and the initial conditions:%
\begin{equation*}
p_{0}^{c}\left( 0,0\right) =1,~p_{k}^{c}\left( 0,0\right) =p_{k}^{c}\left(
t_{1},0\right) =p_{k}^{c}\ \left( 0,t_{2}\right) =0,\ k\geq 1.  %\label{5.4}
\end{equation*}

We are now ready to derive the governing equations of the marginal
distributions of FPRF $N_{\alpha _{1},\alpha
_{2}}(t_{1},t_{2}),~(t_{1},t_{2})\in \mathbb{R}_{+}^{2}:$%
\begin{equation}
p_{k}^{\alpha _{1},\alpha _{2}}\left( t_{1},t_{2}\right) =\mathrm{P}\left(
N_{\alpha _{1},\alpha _{2}}(t_{1},t_{2})=k\right) ,\ k=0,1,2,\ldots
\label{MD}
\end{equation}%
given by (\ref{P2}) or (\ref{LP}). These equations have something in common
with the governing equations for the non-homogeneous Fractional Poisson
processes \cite{LST}.

For a function $u(t_{1},t_{2}),~(t_{1},t_{2})\in \mathbb{R}_{+}^{2},$ the
Caputo-Djrbashian mixed fractional derivative of order $\alpha _{1},\alpha
_{2}\in (0,1)\times \left( 0,1\right) $ is defined by%
\begin{multline*}
\mathrm{D}_{t_{1},t_{2}}^{\alpha _{1},\alpha _{2}}u(t_{1},t_{2}) 
=\frac{1}{%
\Gamma \left( 1-\alpha _{1}\right) \Gamma \left( 1-\alpha _{2}\right) }%
\int_{0}^{t_{1}}\int_{0}^{t_{2}}\frac{\partial ^{2}u\left( \tau _{1},\tau
_{2}\right) }{\partial \tau _{1}~\partial \tau _{2}}\frac{d \tau
_{1}\,d\tau _{2} }{\left( t_{1}-\tau _{1}\right) ^{\alpha _{1}}\left(
t_{2}-\tau _{2}\right) ^{\alpha _{2}}}  
\\
=\frac{1}{\Gamma \left( 1-\alpha _{1}\right) \Gamma \left( 1-\alpha
_{2}\right) }\int_{0}^{t_{1}}\int_{0}^{t_{2}}\frac{\partial ^{2}u\left(
t_{1}-\upsilon _{1},t_{2}-\upsilon _{2}\right) }{\partial \upsilon
_{1}~\partial \upsilon _{2}}\frac{d \upsilon _{1}\,d\upsilon
_{2}}{\upsilon _{1}^{\alpha _{1}}\upsilon _{2}^{\alpha _{2}}}.  
\end{multline*}

Assuming that 
\begin{equation*}
e^{-s_{1}t_{1}-s_{2}t_{2}}\frac{\partial ^{2}u\left( t_{1}-\upsilon
_{1},t_{2}-\upsilon _{2}\right) }{\partial \upsilon _{1}~\partial \upsilon
_{2}}\upsilon _{1}^{-\alpha _{1}}\upsilon _{2}^{-\alpha _{2}}
\end{equation*}%
is integrable as function of four variables $t_{1},t_{2}$,$\upsilon
_{1},\upsilon _{2},$ the double Laplace transform of the the Caputo-Djrbashian mixed fractional derivative%
\begin{multline}
\mathcal{L}\left\{ \mathrm{D}_{t_{1},t_{2}}^{\alpha _{1},\alpha
_{2}}u(t_{1},t_{2});s_{1},s_{2}\right\} =\int_{0}^{\infty }\int_{0}^{\infty
}e^{-s_{1}t_{1}-s_{2}t_{2}}\mathrm{D}_{t_{1},t_{2}}^{\alpha _{1},\alpha
_{2}}u(t_{1},t_{2})dt_{1}~dt_{2}
\\
=s_{1}^{\alpha _{1}}s_{2}^{\alpha _{2}}\tilde{u}(s_{1},s_{2})-s_{1}^{\alpha
_{1}-1}s_{2}^{\alpha _{2}}\tilde{u}(s_{1},0)-s_{1}^{\alpha
_{1}}s_{2}^{\alpha _{2}-1}\tilde{u}(0,s_{2})-s_{1}^{\alpha
_{1}-1}s_{2}^{\alpha _{2}-1}\tilde{u}(0,0),  \label{5.6}
\end{multline}%
where $\tilde{u}(s_{1},s_{2})=\mathcal{L}\left\{
u(t_{1},t_{2});s_{1},s_{2}\right\} $ is the double Laplace transform of the
function $u(t_{1},t_{2})$. 

\begin{remark}%
Note that the Laplace transform of $f_{\alpha}(t,x)$ given by (\ref{L1}) as $\alpha =1$
is of the form $e^{-sx}$ and its inverse is the delta distribution $\delta(t-x)$. Accordingly,
as $\alpha \to 1$, $f_{\alpha}(t,x)$ converges weakly to $\delta(t-x)$, and we denote it 
by $f_{\alpha}(t,x)\to \delta(t-x)$.
\end{remark}%

The proof of (\ref{5.6}) is standard and we omit it (see \cite[p. 37]{MeS}
for the one-dimensional case).

\begin{theorem}
Let $N(t_{1},t_{2}),\ (t_{1},t_{2})\in \mathbb{R}_{+}^{2},\alpha _{1},\alpha
_{2}\in (0,1)\times \left( 0,1\right) ,$ be the FPRF defined by (\ref{FPRF}).

\textbf{1)} Then its marginal distribution given in (\ref{MD}) satisfy the
following fractional differential-integral recurrent equations:%
\begin{align}
\mathrm{D}_{t_{1},t_{2}}^{\alpha _{1},\alpha _{2}}p_{0}^{\alpha _{1},\alpha
_{2}}\left( t_{1},t_{2}\right) 
&=\int_{0}^{\infty }\int_{0}^{\infty }\left(
-\lambda +\lambda ^{2}x_{1}x_{2}\right) p_{0}^{\alpha _{1},\alpha
_{2}}\left( x_{1},x_{2}\right)  f_{\alpha _{1}}(t_{1},x_{1})f_{\alpha
_{2}}(t_{2},x_{2})dx_{1}dx_{2};  \label{5.7}
\\
\mathrm{D}_{t_{1},t_{2}}^{\alpha _{1},\alpha _{2}}p_{1}^{\alpha _{1},\alpha
_{2}}\left( t_{1},t_{2}\right) &
=\int_{0}^{\infty }\int_{0}^{\infty }\big[ \left( -3\lambda +\lambda
^{2}x_{1}x_{2}\right) p_{1}^{\alpha _{1},\alpha _{2}}\left(
x_{1},x_{2}\right) \notag \\
& \qquad +\lambda p_{0}^{\alpha _{1},\alpha _{2}}\left(
x_{1},x_{2}\right) \big] 
f_{\alpha _{1}}(t_{1},x_{1})f_{\alpha
_{2}}(t_{2},x_{2})dx_{1}dx_{2};  \label{5.8}
\\
\mathrm{D}_{t_{1},t_{2}}^{\alpha _{1},\alpha _{2}}p_{k}^{\alpha _{1},\alpha
_{2}}\left( t_{1},t_{2}\right) 
& 
=\int_{0}^{\infty }\int_{0}^{\infty }\Big[ \left( -\lambda +\lambda
^{2}x_{1}x_{2}\right) p_{k}^{\alpha _{1},\alpha _{2}}\left(
x_{1},x_{2}\right) \notag \\
& \qquad +\left( \lambda -2\lambda ^{2}x_{1}x_{2}\right)
p_{k-1}^{\alpha _{1},\alpha _{2}}\left( x_{1},x_{2}\right) +\lambda
^{2}x_{1}xp_{k-2}^{\alpha _{1},\alpha _{2}}\left( x_{1},x_{2}\right) \Big]
\notag 
\\
& \qquad \qquad \times f_{\alpha _{1}}(t_{1},x_{1})f_{\alpha
_{2}}(t_{2},x_{2})dx_{1}dx_{2}, \qquad\qquad\qquad\qquad k\geq 2;  \label{5.9}
\end{align}

with the initial conditions:%
\begin{equation*}
p_{0}^{\alpha _{1},\alpha _{2}}(0,0)=1,~p_{k}^{\alpha _{1},\alpha
_{2}}(0,0)=p_{k}^{\alpha _{1},\alpha _{2}}\left( t_{1},0\right)
=p_{k}^{\alpha _{1},\alpha _{2}}\left( 0,t_{2}\right) =0,~k\geq 1.
%\label{5.10}
\end{equation*}

\textbf{2)} For $\alpha _{1}\rightarrow 1,\alpha _{2}\rightarrow 1,f_{\alpha
_{1}}(t_{1},x_{1})\rightarrow \delta (t_{1}-x_{1}),\ f_{\alpha
_{2}}(t_{2},x_{2})\rightarrow \delta (t_{2}-x_{2}),$ hence (\ref{5.7}),\ (%
\ref{5.8}) and (\ref{5.9}) become (\ref{5.1}), (\ref{5.2}) and (\ref{5.3})
correspondingly.
\end{theorem}

\begin{proof}
\textbf{1)} The initial conditions are easily checked using the fact that $Y_{\alpha
_{1}}(0)=Y_{\alpha _{2}}(0)=0$ a.s.

Let $p_{k}^{\alpha _{1},\alpha _{2}}\left( t_{1},t_{2}\right)
,k=0,1,2,\ldots $, be defined as in equations (\ref{P2}) or (\ref{LP}). Then
the characteristic function of the FPRF, for $z\in \mathbb{R}$:%
\begin{equation}
\hat{p}(t_{1},t_{2};z)=\mathrm{E}\exp \left\{ izN_{\alpha _{1},\alpha
_{2}}(t_{1},t_{2})\right\} 
=\int_{0}^{\infty }\int_{0}^{\infty }e^{\lambda
x_{1}x_{2}(e^{iz}-1)}f_{\alpha _{1}}(t_{1},x_{1})f_{\alpha
_{2}}(t_{2},x_{2})dx_{1}dx_{2}.  \label{5.11}
\end{equation}

Taking the double Laplace transform of (\ref{5.11}) and using (\ref{L1}) and
(\ref{L2}) yields%
\begin{align}
\bar{p}(s_{1},s_{2};z) &=\widetilde{\hat{p}}(t_{1},t_{2};z)=\int_{0}^{%
\infty }\int_{0}^{\infty }e^{-s_{1}t_{1}-s_{2}t_{2}}\hat{p}%
(t_{1},t_{2};z)dt_{1}dt_{2}  \label{5.12} \\
&=s_{1}^{\alpha _{1}-1}s_{2}^{\alpha _{2}-1}\int_{0}^{\infty
}\int_{0}^{\infty }e^{\lambda x_{1}x_{2}(e^{iz}-1)}e^{-x_1s_{1}^{\alpha
_{1}}-x_2s_{2}^{\alpha _{2}}} 
%f_{\alpha _{1}}(t_{1},x_{1})f_{\alpha
%_{2}}(t_{2},x_{2})
dx_{1}dx_{2},  \notag
\end{align}%
and%
\begin{equation*}
\bar{p}\left( 0,0,z\right) =\bar{p}\left( 0,s_{2},z\right) =\bar{p}\left(
s_{1},0,z\right) =0.  %\label{5.13}
\end{equation*}

Using an integration by parts for a double integral \cite{LMS00}:%
\begin{align*}
\int_{0}^{\infty }\int_{0}^{\infty }&F(x_{1},x_{2})H\left(
dx_{1},dx_{2}\right) =\int_{0}^{\infty }\int_{0}^{\infty }H\left(
[x_{1},\infty )\times [ x_{2},\infty )\right) F\left(
dx_{1},dx_{2}\right)
\\
& \qquad +\int_{0}^{\infty }H\left( [x_{1},\infty )\times [ 0,\infty )\right)
F\left( dx_{1},0\right) 
\\
& \qquad +\int_{0}^{\infty }H\left( [0,\infty )\times [
x_{2},\infty )\right) F\left( 0,dx_{2}\right) +F(0,0)H\left( [0,\infty
)\times [ 0,\infty )\right) ,  %\label{5.15}
\end{align*}%
we get from (\ref{5.6}), (\ref{5.12}) and (\ref{5.12}) with 
\begin{align*}
F(x_{1},x_{2})&=\exp \left\{ \lambda x_{1}x_{2}(e^{iz}-1)\right\} ,\ H\left(
dx_{1},dx_{2}\right) =\exp \left\{ -s_{1}^{\alpha _{1}}x_{1}-s_{2}^{\alpha
_{2}}x_{2}\right\} dx_{1}dx_{2},
\\
\bar{p}(s_{1},s_{2};z)&=s_{1}^{\alpha _{1}-1}s_{2}^{\alpha _{2}-1}\Big[
\int_{0}^{\infty }\int_{0}^{\infty }\frac{\partial ^{2}\exp \left\{
ix_{1}x_{2}(e^{iz}-1)\right\} }{\partial x_{1}~\partial x_{2}} \\
& \qquad  \qquad  \qquad  \qquad  \times \frac{\exp
\left\{ -s_{1}^{\alpha _{1}}x_{1}-s_{2}^{\alpha _{2}}x_{2}\right\} }{%
s_{1}^{\alpha _{1}}s_{2}^{\alpha _{2}}}dx_{1},dx_{2}
+\frac{\hat{p}(0,0,z)}{%
s_{1}^{\alpha _{1}}s_{2}^{\alpha _{2}}}\Big] .
\end{align*}%
Thus%
\begin{multline*}
s_{1}^{\alpha _{1}}s_{2}^{\alpha _{2}}\bar{p}(s_{1},s_{2};z)-\hat{p}%
(0,0,z)\\
=s_{1}^{\alpha _{1}-1}s_{2}^{\alpha _{2}-1}\int_{0}^{\infty
}\int_{0}^{\infty }\frac{\partial ^{2}\exp \left\{
ix_{1}x_{2}(e^{iz}-1)\right\} }{\partial x_{1}~\partial x_{2}}\exp \left\{
-s_{1}^{\alpha _{1}}x_{1}-s_{2}^{\alpha _{2}}x_{2}\right\} dx_{1},dx_{2}
\end{multline*}

Using (\ref{5.6}), (\ref{L1}) we can invert the double Laplace transform as
follows:%
\begin{equation*}
\mathrm{D}_{t_{1},t_{2}}^{\alpha _{1},\alpha _{2}}\hat{p}\left(
t_{1},t_{2},z\right) =\int_{0}^{\infty }\int_{0}^{\infty }\frac{\partial
^{2}\exp \left\{ ix_{1}x_{2}(e^{iz}-1)\right\} }{\partial x_{1}~\partial
x_{2}}f_{\alpha _{1}}(t_{1},x_{1})f_{\alpha _{2}}(t_{2},x_{2})dx_{1}dx_{2}.
\end{equation*}%
And finally, by inverting the characteristic function (\ref{5.11}), we
obtain 
\begin{equation*}
\mathrm{D}_{t_{1},t_{2}}^{\alpha _{1},\alpha _{2}}\hat{p}\left(
t_{1},t_{2},z\right) p_{k}^{\alpha _{1},\alpha _{2}}\left(
t_{1},t_{2}\right) =\int_{0}^{\infty }\int_{0}^{\infty }\left[ \frac{%
\partial ^{2}}{\partial x_{1}~\partial x_{2}}p_{k}^{c}(x_{1},x_{2})\right]
f_{\alpha _{1}}(t_{1},x_{1})f_{\alpha _{2}}(t_{2},x_{2})dx_{1}dx_{2}.
\end{equation*}

Using (\ref{5.1}), (\ref{5.2}) and (\ref{5.3}) we arrive to (\ref{5.7}),\ (%
\ref{5.8}) and (\ref{5.9}) correspondingly.

\medskip

\noindent \textbf{2)} 
Finally, as $\alpha _{j}\rightarrow 1,j=1,2$ we have $e^{-s_{j}^{\alpha
_{j}}x_{j}}\rightarrow e^{-s_{j}x_{j}},j=1,2$, and their Laplace inversions
are delta function: $\delta (t_{j}-x_{j}),j=1,2$.
Thus, 2) is proven.
\end{proof}

\section{Simulations}

In this section we show some simulations of FPRF made with Matlab based on the 
$\alpha$-stable random number generator function 
{\texttt stblrnd}. 
For a relevant work on statistical parameter estimation of FPP 
in connection with simulations, see also \cite{Cahoy10}.

\begin{figure}
\centering
\subfigure[$\alpha_1 = 0.95$, $\alpha_2 = 0.5$, $\lambda = 100$]{
\includegraphics[width=.45\textwidth]{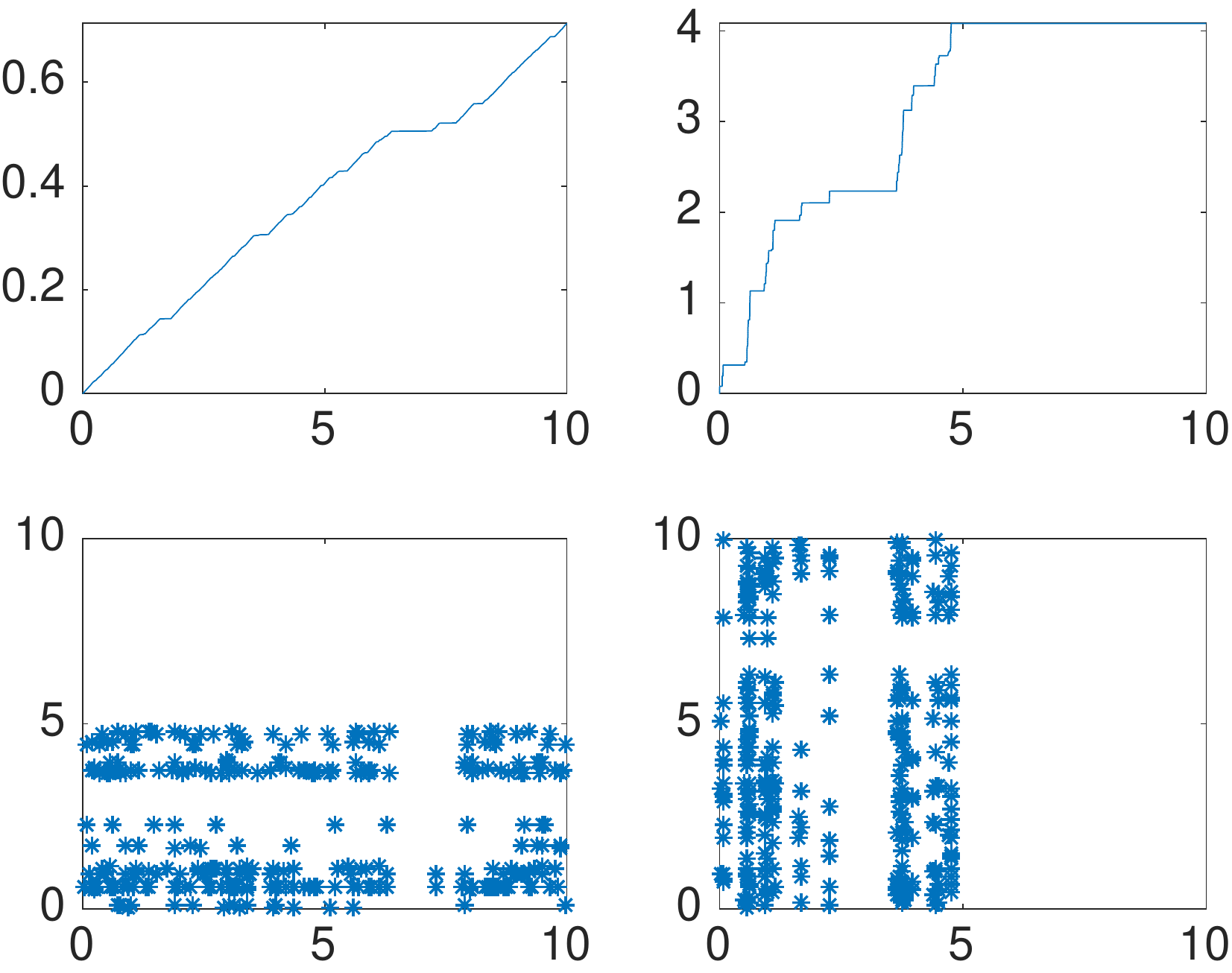}}
\qquad
\subfigure[$\alpha_1 = \alpha_2 = 0.75$, $\lambda = 100$]{
\includegraphics[width=.45\textwidth]{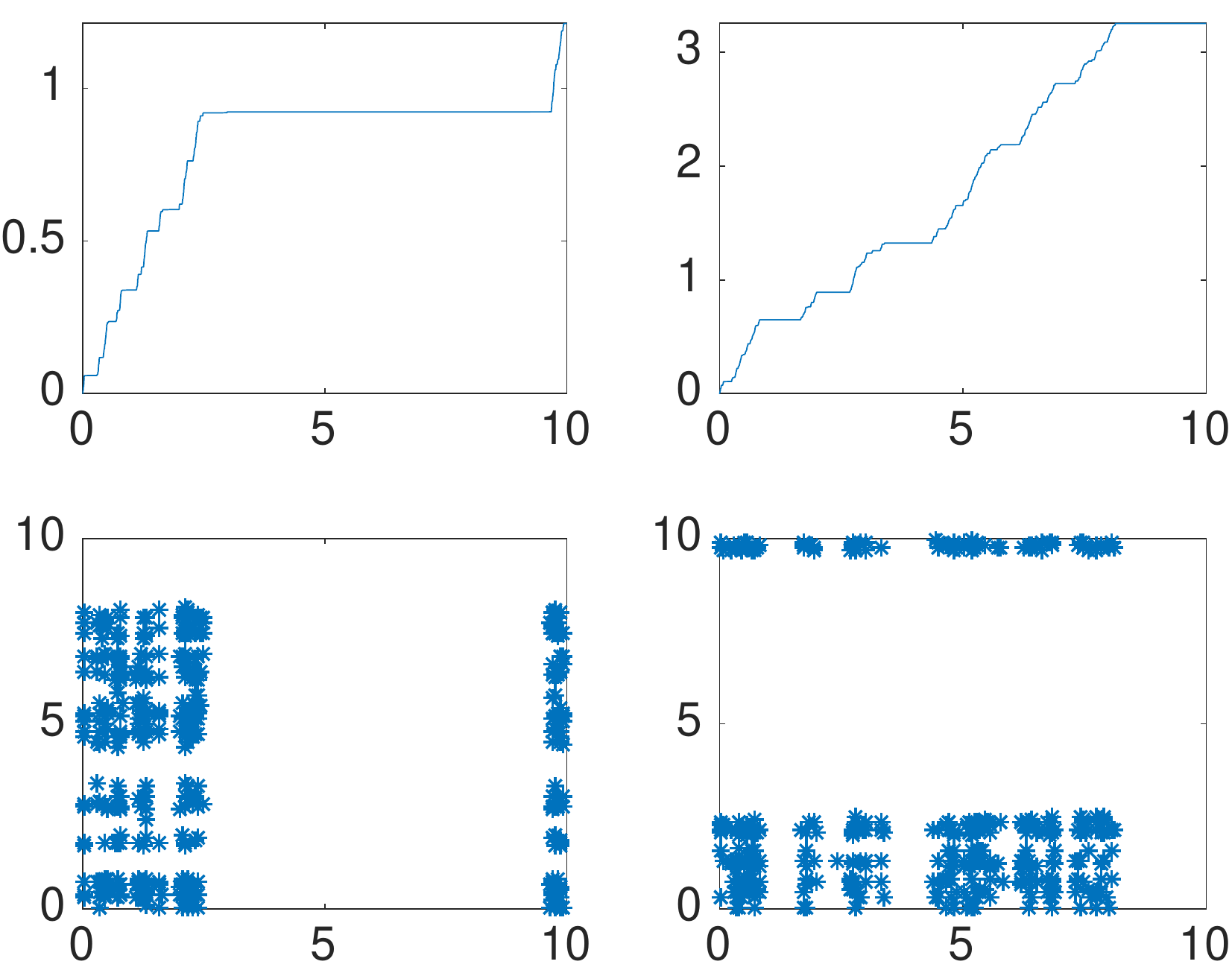}}
\\
\subfigure[$\alpha_1 = 0.9$, $\alpha_2 = 0.75$, $\lambda = 100$]{
\includegraphics[width=.45\textwidth]{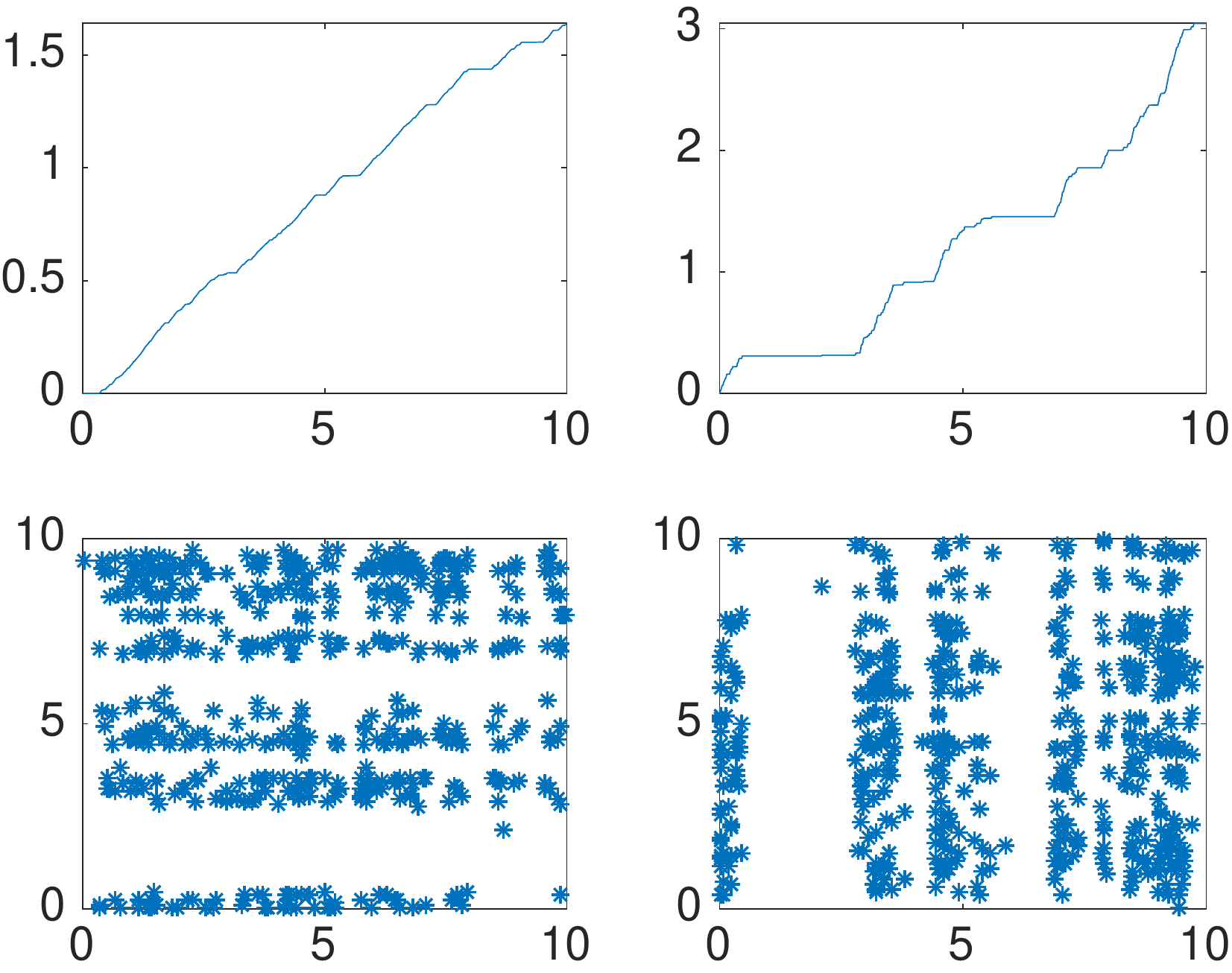}}
\qquad
\subfigure[$\alpha_1 =\alpha_2 \sim 1$, $\lambda = 10000$]{
\includegraphics[width=.45\textwidth]{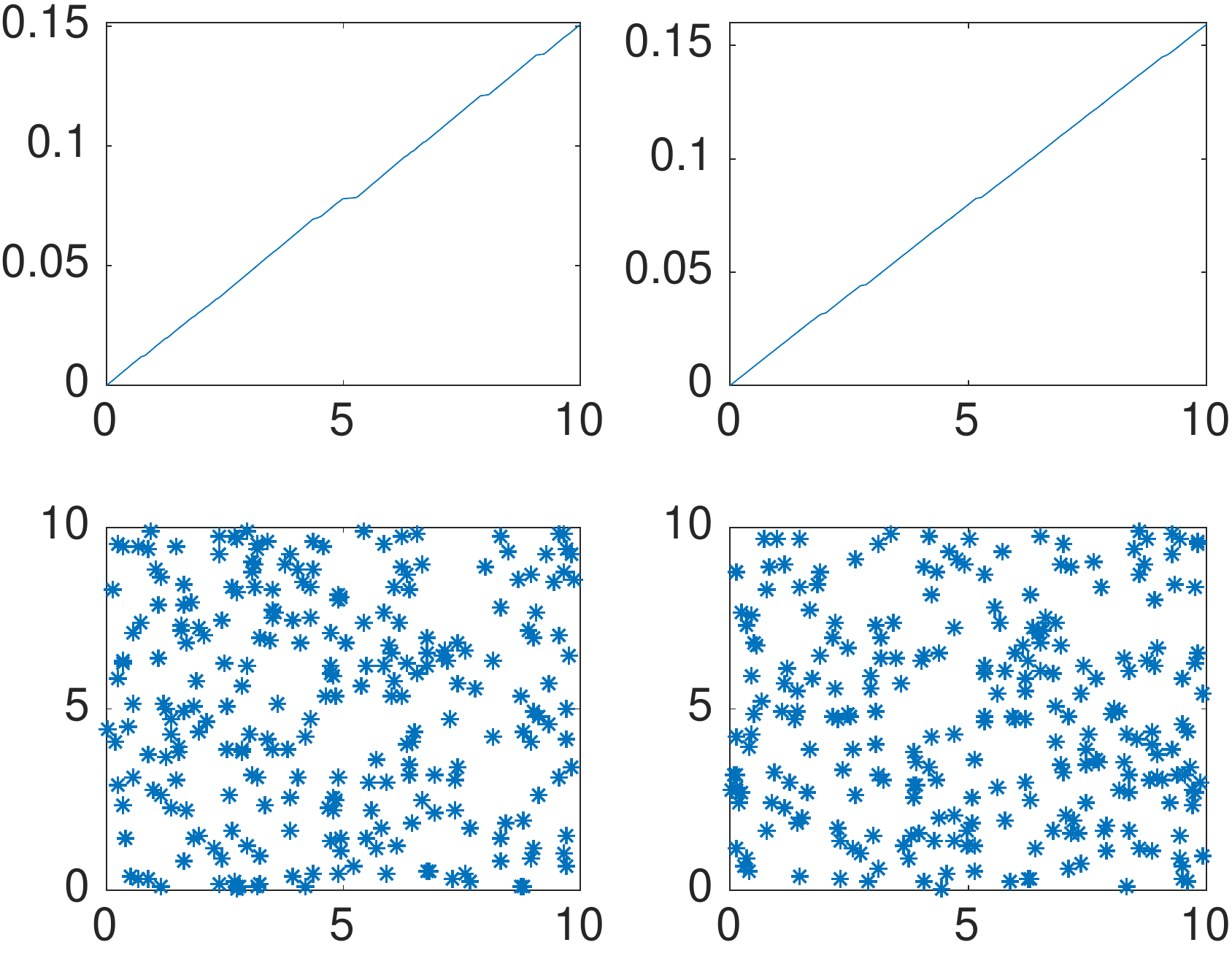}}
%\begin{subfigure}[b]{0.45\textwidth}
%        \fbox{\includegraphics[width=.95\textwidth]{simul_95_50_lambda_100.pdf}}
%        \caption{$\alpha_1 = 0.95$, $\alpha_2 = 0.5$, $\lambda = 100$}
% %       \label{fig1a}
%    \end{subfigure}
%\begin{subfigure}[b]{0.45\textwidth}
%        \fbox{\includegraphics[width=.95\textwidth]{simul_75_75_lambda_100.pdf}}
%        \caption{$\alpha_1 = \alpha_2 = 0.75$, $\lambda = 100$}
%  %      \label{fig1b}
%    \end{subfigure}
%\\
%\begin{subfigure}[b]{0.45\textwidth}
%        \fbox{\includegraphics[width=.95\textwidth]{simul_90_75_lambda_100.pdf}}
%        \caption{$\alpha_1 = 0.9$, $\alpha_2 = 0.75$, $\lambda = 100$}
% %       \label{fig1c}
%    \end{subfigure}
%\begin{subfigure}[b]{0.45\textwidth}
%        \fbox{\includegraphics[width=.95\textwidth]{simul_99_99_lambda_10000.pdf}}
%        \caption{$\alpha_1 =\alpha_2 \sim 1$, $\lambda = 10000$}
%%        \label{fig1d}
%    \end{subfigure}
\caption{Simulations of the inverse stable subordinators $Y_{\alpha _{1}}^{(1)}(t)$ and $Y_{\alpha _{2}}^{(2)}(t)$
and the corresponding FPRF $N_{\alpha_1,\alpha_2}$ for different values of $\alpha_1$ and $\alpha_2$.
Top-left: simulation of $Y_{\alpha _{1}}^{(1)}(t)$, top-right: simulation of $Y_{\alpha _{2}}^{(2)}(t)$, 
bottom-(left-right): simulation of $N_{\alpha_1,\alpha_2}$, the rotation shows the connection with marginal intensity}\label{fig:1}
\end{figure}

The subordinators $L_{\alpha}$ are simulated exactly at times $t_n = n \Delta$, where
$\Delta = 0.0005$ till they reach a defined value $S_{\mathrm{end}}$.
%STBLRND alpha-stable random number generator..
More precisely, 
$$
L_{\alpha}(0) = 0; \qquad L_{\alpha}(t_n) = L_{\alpha}(t_{n-1}) + X  , \qquad n= 1, 2, \ldots , N
$$ 
where $X$ is independently simulated with {\texttt stblrnd($\alpha$, 1, $\sqrt[\alpha]{\Delta}$ , 0)}.
Accordingly,
\[
\mathrm{E}e^{-sX}=\exp \{-(s \sqrt[\alpha]{\Delta})^{\alpha }\} = 
\exp \{-\Delta s^{\alpha }\} ,\qquad s\geq 0,
\]
and hence
\[
\mathrm{E}e^{-sL_{\alpha}(t_n)} = 
\exp \{-t_ns^{\alpha }\} ,\qquad s\geq 0, n= 0,1,\ldots, N.
\]
% The limitation here is $\alpha <1$, as expected.
The simulation of the inverse stable subordinators $Y_{\alpha}(s), s\in [0, T_{\mathrm{end}}]$
are thus made at times $s_n = L_{\alpha}(t_{n}), n=1,\ldots, N$ with values $Y_{\alpha}(s_n) = n \Delta$.

To simulate a FPRF $N_{\alpha_1,\alpha_2}(s^{1},s^{2})$ on the window
$(0,S_{\mathrm{end}})\times(0,S_{\mathrm{end}})$, we first simulate two independent inverse stable subordinators 
$Y_{\alpha _{1}}^{(1)}(s^{1}_n),n=1,\ldots, N_1$ and $Y_{\alpha _{2}}^{(2)}(s^{2}_n),n=1,\ldots, N_2 $.

By Proposition~\ref{thm:char1}, the value of $N_{\alpha_1,\alpha_2}$ 
on each rectangle $(s^{1}_n,s^{1}_{n+1})\times (s^{1}_n,s^{1}_{n+1})$ is a Poisson random variable
with mean $\Delta^2$. As $\Delta^2\ll 1$, we approximate it with a Bernoulli random variable $Y$ of parameter $\Delta^2$.
When $Y=1$, we add a point at random inside the rectangle.

In Figure~\ref{fig:1} the simulations 
of the inverse stable subordinators $Y_{\alpha _{1}}^{(1)}(t)$ and $Y_{\alpha _{2}}^{(2)}(t)$
and the corresponding FPRF $N_{\alpha_1,\alpha_2}$ for different values of $\alpha_1$ and $\alpha_2$ are shown.
The simulations of $N_{\alpha_1,\alpha_2}$ are plotted twice: we have rotated each figure in order to underline
the spatial dependence of the spread of the points of the process $N_{\alpha_1,\alpha_2}$ 
in connection with the marginal intensities $Y_{\alpha _{1}}^{(1)}(t)$ and $Y_{\alpha _{2}}^{(2)}(t)$.
For example, in Figure~\ref{fig:1}(c) two different marginal distribution are expected since $\alpha_1=0.9$ and
$\alpha_2=0.75$. While $Y_{0.9}^{(1)}(t)$ produces a quite uniform distribution of points, $Y_{0.75}^{(2)}(t)$ generates
clusters in correspondence of its steeper slopes.

We also compute the quantity
$$
\mathrm{P}\left( N(Y_{1}(t_{1}), Y_{2}(t_{2}))=k\right) = \int_{0}^{\infty }\int_{0}^{\infty }\frac{e^{-\lambda x_{1}x_{2}}(\lambda
x_{1}x_{2})^{k}}{k!}f_{\alpha _{1}}(t_{1},x_{1})f_{\alpha
_{2}}(t_{2},x_{2})dx_{1}dx_{2},
$$ 
given in \eqref{P2}, for different values of $t_1,t_2,\alpha_1$ and $\alpha_2$. In fact, with a Monte Carlo procedure, 
we approximate the above quantity with 
%$f_{\alpha _{1}}(t_{1},x_{1})f_{\alpha_{2}}(t_{2},x_{2})$ with
\[
\frac{1}{N^2}\sum_{n_1=1}^N \sum_{n_2=1}^N
\frac{e^{-\lambda x_{1}x_{2}}(\lambda
x_{1}x_{2})^{k}}{k!}
{\indI}_{X_{n_1}} (x_{1}) {\indI}_{Y_{n_2}} (x_{2})
\]
where $(X_{n},n=1,\ldots,N)$ and $(Y_{n},n=1,\ldots,N)$ are independent sequences of i.i.d.\ distributed as
$Y_{\alpha _{1}}^{(1)}(t_1)$ and $Y_{\alpha _{2}}^{(2)}(t_2)$, respectively. 
Summing up, the integral in \eqref{P2} is computed numerically, and
the simulations with $N=1500$ are
presented in  Figure~\ref{fig:2}. We underline the variety of the shape of distributions that can be generated with this two-parameter model in addition to its flexibility to include, for example, different cluster phenomena.

\begin{figure}
\centering
\subfigure[$t_1=t_2=5$,  $\alpha_1 = 0.5$, $\alpha_2 = .75$]{
\includegraphics[width=.45\textwidth]{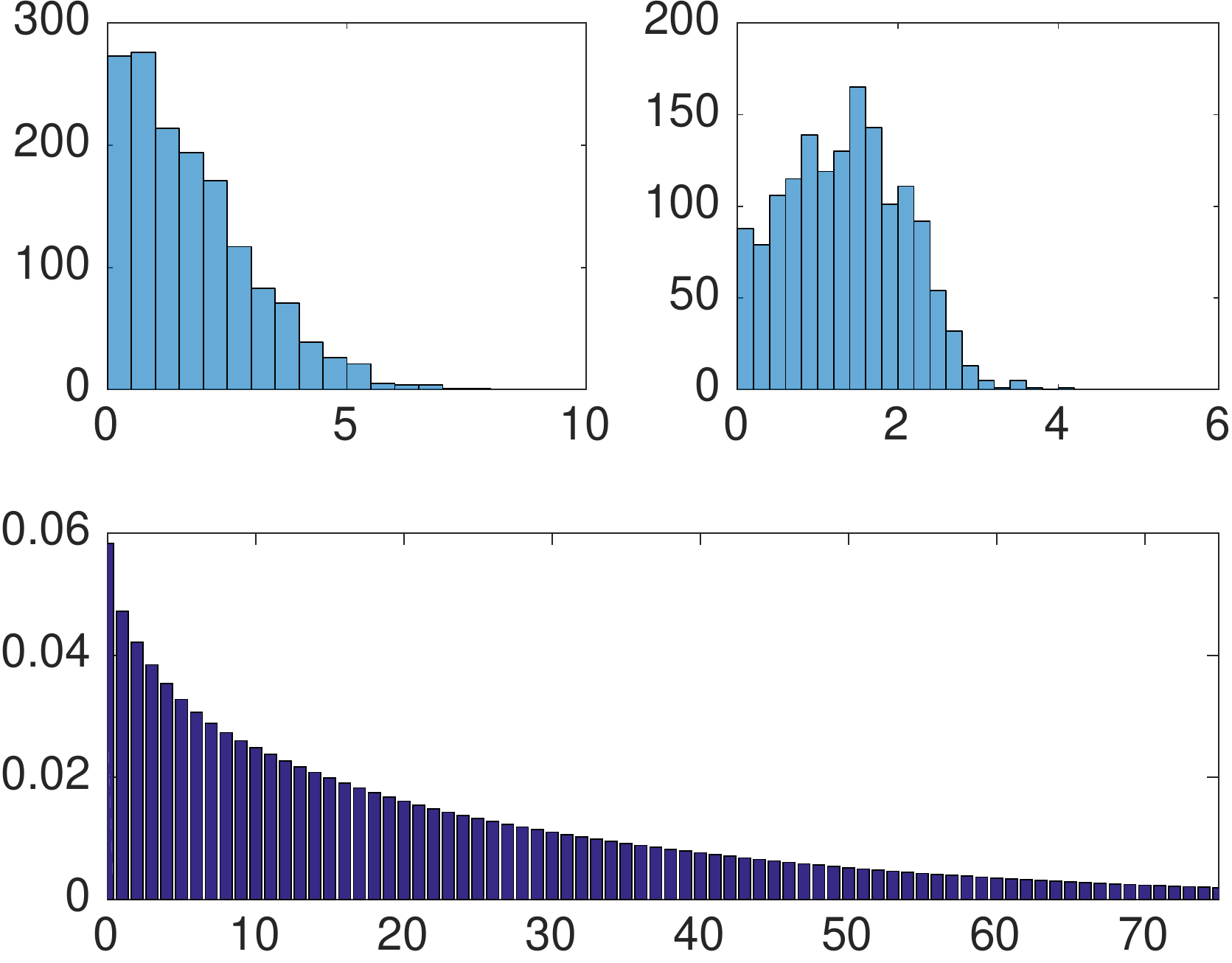}}
\qquad
\subfigure[$t_1=t_2=5$,  $\alpha_1 = 0.95$, $\alpha_2 = .075$]{
\includegraphics[width=.45\textwidth]{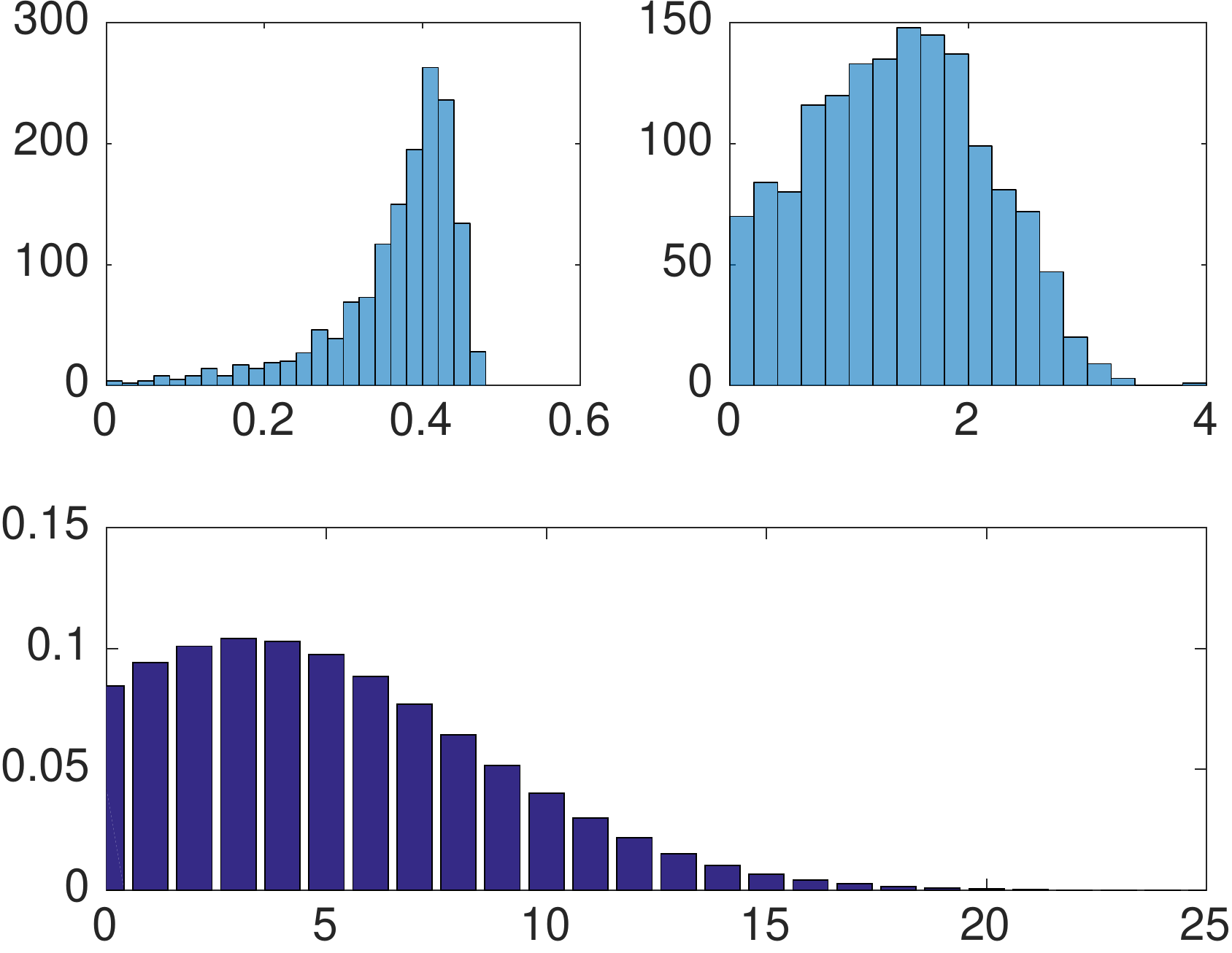}}
\\
\subfigure[$t_1=t_2=5$,  $\alpha_1 = \alpha_2 = 0.75$]{
\includegraphics[width=.45\textwidth]{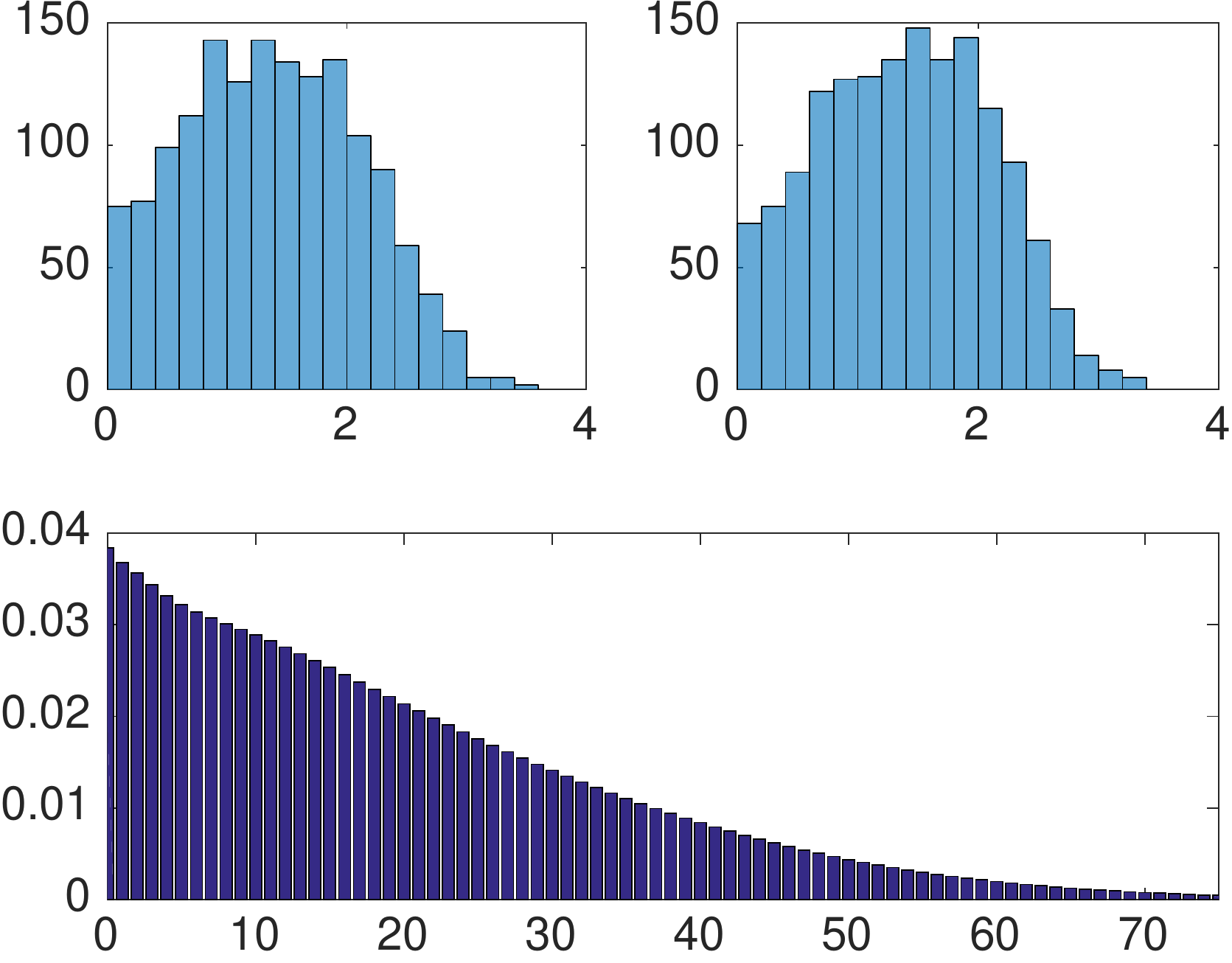}}
\qquad
\subfigure[$t_1=3$, $t_2=7$,  $\alpha_1 = \alpha_2 = 0.75$]{
\includegraphics[width=.45\textwidth]{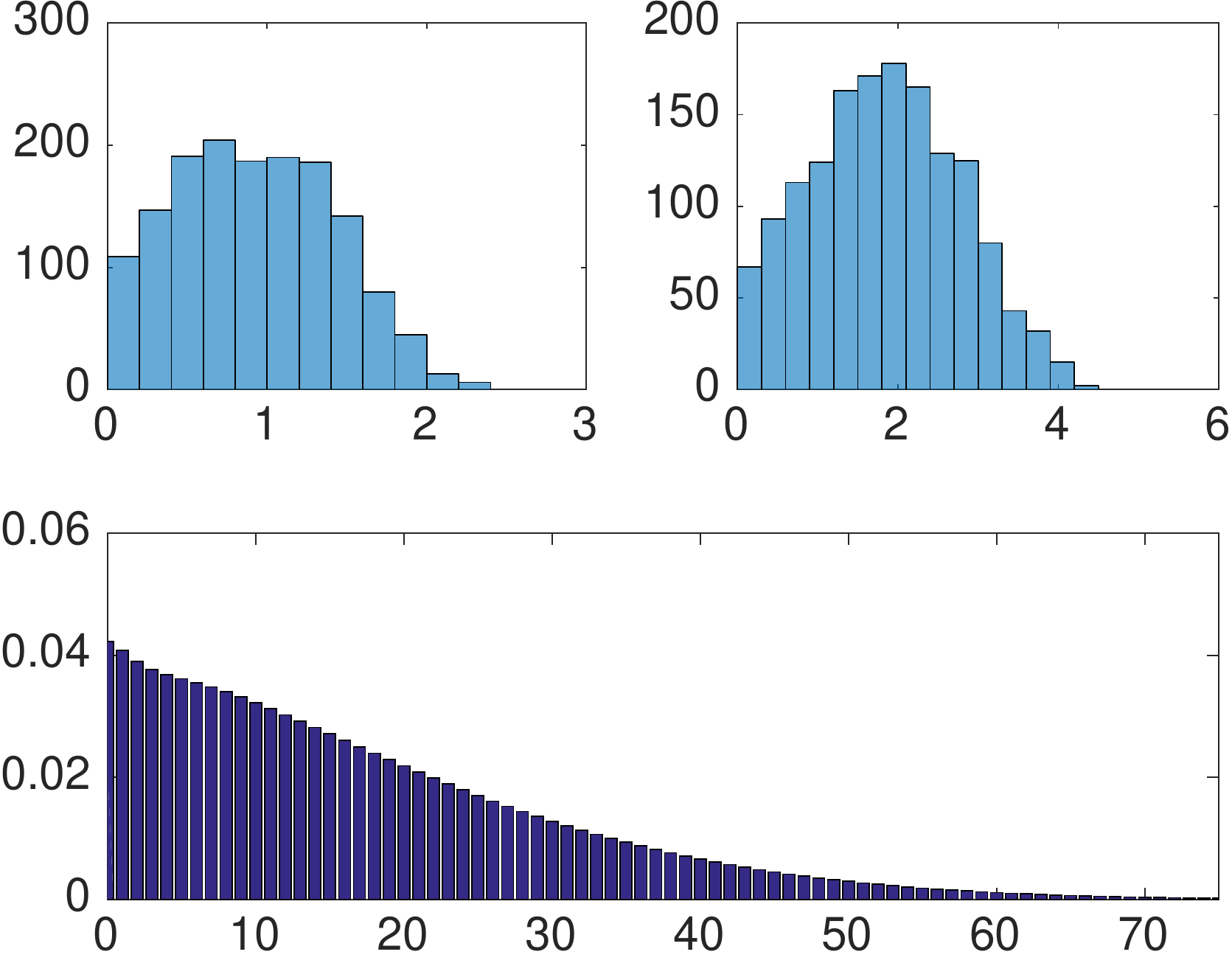}}
%\begin{subfigure}[b]{0.45\textwidth}
%        \fbox{\includegraphics[width=.95\textwidth]{Sim_T_5_5_al_50_75_lam_10.pdf}}
%        \caption{$t_1=t_2=5$,  $\alpha_1 = 0.5$, $\alpha_2 = .75$}
% %       \label{fig1a}
%    \end{subfigure}
%\begin{subfigure}[b]{0.45\textwidth}
%        \fbox{\includegraphics[width=.95\textwidth]{Sim_T_5_5_al_95_75_lam_10.pdf}}
%        \caption{$t_1=t_2=5$,  $\alpha_1 = 0.95$, $\alpha_2 = .075$}
%  %      \label{fig1b}
%    \end{subfigure}
%\\
%\begin{subfigure}[b]{0.45\textwidth}
%        \fbox{\includegraphics[width=.95\textwidth]{Sim_T_5_5_al_75_75_lam_10.pdf}}
%        \caption{$t_1=t_2=5$,  $\alpha_1 = \alpha_2 = 0.75$}
% %       \label{fig1c}
%    \end{subfigure}
%\begin{subfigure}[b]{0.45\textwidth}
%        \fbox{\includegraphics[width=.95\textwidth]{Sim_T_3_7_al_75_75_lam_10.pdf}}
%        \caption{$t_1=3$, $t_2=7$,  $\alpha_1 = \alpha_2 = 0.75$}
%%        \label{fig1d}
%    \end{subfigure}
\caption{Simulations of the distribution of $Y_{\alpha _{1}}^{(1)}(t_1)$, $Y_{\alpha _{2}}^{(2)}(t_2)$
and the corresponding $p_{k}(t_{1},t_{2}) = \mathrm{P}\left( N(Y_{1}(t_{1}), Y_{2}(t_{2}))=k\right) $ 
for $\lambda = 10$ and different values of $t_1,t_2,\alpha_1$ and $\alpha_2$.}\label{fig:2}
\end{figure}

\begin{small}
\section*{Acknowledgement}
N. Leonenko and E. Merzbach wish to thank G. Aletti for two visits to University of Milan
\end{small}

\appendix

\section{Covariance Structure of Parameter-Changed Poisson random
fields}

In this Appendix, we prove a general result that can be used to compute the
covariance structure of the parameter-changed Poisson random field: 
\begin{equation*}
Z\left( t_{1},t_{2}\right) =N(Y_{1}(t_{1}),Y_{2}(t_{2})),\ (t_{1},t_{2})\in 
\mathbb{R}_{+}^{2},  %\label{3.1}
\end{equation*}
where $Y_{1}=\left\{ Y_{1}(t_{1}),t_{1}\geq 0\right\} $ and $Y_{2}=\left\{
Y_{2}(t_{2}),t_{2}\geq 0\right\} $ are independent non-negative
non-decreasing stochastic processes, in general non-Markovian with
non-stationary and non-independent increments, and $N=%
\{N(t_{1},t_{2}),(t_{1},t_{2})\in \mathbb{R}_{+}^{2}\}$ is a PRF with
intensity $\lambda >0$. We also assume that $Y_{1}$ and $Y_{2}$ are
independent of $N.$

For example, $Y_{1}$ and $Y_{2}$ might be inverse subordinators.

\begin{theorem}\label{thm:A1}
Suppose that $N$ is a PRF, $Y_{1}$ and $Y_{2}$ are two non-decreasing
non-negative independent stochastic processes which are also independent of $%
N. $ Then

1) if $\mathrm{E}Y_{1}(t_{1})=U_{1}(t_{1})$\ and $\mathrm{E}%
Y_{2}(t_{2})=U_{2}(t_{2})$ exist, then $\mathrm{E}Z(t_{1},t_{2})$ exists and 
\begin{equation*}
\mathrm{E}Z(t_{1},t_{2})=\mathrm{E}N(1,1)\mathrm{E}Y_{1}(t_{1})\mathrm{E}%
Y_{2}(t_{2}); % \label{3.2}
\end{equation*}

2) if $Y_{1}$ and $Y_{2}$ have second moments, so does $Z$ and 
\begin{gather*}
\mathrm{Var}Z(t_{1},t_{2})=\left[ \mathrm{E}N(1,1)\right] ^{2}\left\{ 
\mathrm{E}Y_{1}^{2}(t_{1})\mathrm{E}Y_{2}^{2}(t_{2})-\left( \mathrm{E}%
Y_{1}(t_{1})\right) ^{2}\left( \mathrm{E}Y_{2}(t_{2})\right)
^{2}\right\}  %\label{3.3} 
\\
+\mathrm{Var}N(1,1)\mathrm{E}Y_{1}(t_{1})\mathrm{E}Y_{2}(t_{2})  
\end{gather*}
and its covariance function 
\begin{equation*}
\mathrm{Cov}(Z(t_{1},t_{2}),Z(s_{1},s_{2}))=\mathrm{Cov}\left(
N(Y_{1}(t_{1}),Y_{2}(t_{2})),N(Y_{1}(s_{1}),Y_{2}(s_{2}))\right)   %\tag{A.1}
\end{equation*}
for $s_{1}<t_{1},s_{2}<t_{2}$ is given by: 
\begin{multline} \label{3.5}
(\mathrm{E}N(1,1))^{2} \Big\{ \mathrm{Cov}\left(
Y_{1}(t_{1}),Y_{1}(s_{1})\right) \mathrm{Cov}\left(
Y_{2}(t_{2}),Y_{2}(s_{2})\right) \\
+\mathrm{E}Y_{2}(t_{2})\mathrm{E}Y_{2}(s_{2})  
\mathrm{Cov}\left(Y_{1}(t_{1}),Y_{1}(s_{1})\right) 
+\mathrm{E}Y_{1}(t_{1})\mathrm{E}Y_{1}(s_{1})
\mathrm{Cov}\left(Y_{2}(t_{2}),Y_{2}(s_{2})\right) \Big\} \\
+ \mathrm{Var}N(1,1)\mathrm{E}Y_{1}(s_{1})\mathrm{E}Y_{2}(s_{2})  
\end{multline}
and for any $(s_{1},s_{2}),$ and $(t_{1},t_{2})$ from $\mathbb{R}_{+}^{2}$ 
\begin{multline} \label{3.5b}
(\mathrm{E}N(1,1))^{2} \Big\{ \mathrm{Cov}\left(
Y_{1}(t_{1}),Y_{1}(s_{1})\right) \mathrm{Cov}\left(
Y_{2}(t_{2}),Y_{2}(s_{2})\right) \\
+\mathrm{E}Y_{2}(t_{2})\mathrm{E}Y_{2}(s_{2})  
\mathrm{Cov}\left(Y_{1}(t_{1}),Y_{1}(s_{1})\right) 
+\mathrm{E}Y_{1}(t_{1})\mathrm{E}Y_{1}(s_{1})
\mathrm{Cov}\left(Y_{2}(t_{2}),Y_{2}(s_{2})\right) \Big\} \\
+ \mathrm{Var}N(1,1)\mathrm{E}Y_{1}(\min(s_{1},t_{1}))\mathrm{E}Y_{2}(\min(s_{2},t_{2})) 
\end{multline}
\end{theorem}

%{theorem}

\begin{remark}
These formulae are valid for any L\'{e}vy random field $N=\{N(t_{1},t_{2})$,$%
(t_{1},t_{2})\in \mathbb{R}_{+}^{2}\}$, with finite expectation $\mathrm{E}%
N(1,1)$ and finite variance $\mathrm{Var}N(1,1),$ %of but 
for PRF $\mathrm{E}%
N(1,1)=\lambda ;\ \mathrm{Var}N(1,1)=\lambda $ and to apply these formulae
one needs to know 
\begin{equation*}
U_{1}(t_{1})=\mathrm{E}Y_{1}(t),\ U_{2}(t_{2})=\mathrm{E}Y_{2}(t),\
U_{1}^{(2)}(t_{1})=\mathrm{E}Y_{1}^{2}(t),~U_{2}^{(2)}(t_{1})=\mathrm{E}%
Y_{2}^{2}(t),\ 
\end{equation*}
and $\mathrm{Cov}\left( Y_{1}(t_{1}),Y_{1}(s_{1})\right) ,\ \mathrm{Cov}%
\left( Y_{2}(t_{2}),Y_{2}(s_{2})\right) $ which are available for many
non-negative processes $Y_{1}(t)$ and $Y_{2}(t)$ induction inverse
subordinators.
\end{remark}

\begin{remark}
One can compute the following expression for the one-dimensional
distribution of the parameter-changed PRF: 
\begin{multline*}
\mathrm{P}\left( N(Y_{1}(t_{1}), Y_{2}(t_{2}))=k\right) = p_{k}(t_{1},t_{2})
\\
= \int_{0}^{\infty }\int_{0}^{\infty }\frac{e^{-\lambda x_{1}x_{2}}(\lambda
x_{1}x_{2})^{k}}{k!}f_{1}(t_{1},x_{1})f_{2}(t_{2},x_{2})dx_{1}dx_{2},\
k=0,1,2,\ldots
\end{multline*}
where 
\begin{equation*}
f_{i}(t_{i},x_{i}) = \frac{d}{d{x_i}}\mathrm{P}\left\{ Y_{i}(t_{i})\leq
x_{i}\right\} =\frac{d}{dx_{i}}G_{t_i}^{(i)}(x_{i}), \qquad i=1,2.
% \\
%f_{2}(t_{2},x_{2}) &=\frac{d}{dx}\mathrm{P}\left\{ Y_{2}(t_{2})\leq
%x_{2}\right\} =\frac{d}{dx_{2}}G_{t}^{(2)}(x_{2}),
\end{equation*}
and its Laplace transform: 
\begin{equation*}
\mathcal{L} \left\{ p_{k}(t_{1},t_{2});
s_{1},s_{2} \right\}=\int_{0}^{\infty }\int_{0}^{\infty }\frac{e^{-\lambda
x_{1}x_{2}}(\lambda x_{1}x_{2})^{k}}{k!} \mathcal{L} \left\{
f_{1}(t_{1},x_{1});s_{1}\right\}  \mathcal{L} \left\{
f_{2}(t_{2},x_{2});s_{2}\right\} dx_{1}dx_{2},
\end{equation*}
where 
\begin{equation*}
\mathcal{L} \left\{ f_{i}(t_{i},x_{i});s_{i}\right\} =\int_{0}^{\infty
}e^{-s_{i}t_{i}}f_{i}(t_{i},x_{i})dt_{i},\qquad i=1,2.
\end{equation*}
\end{remark}

\begin{proof}[Proof of Theorem~\ref{thm:A1}]
We denote 
\begin{align*}
& G_{t_{1}}^{(1)}(u_{1})=\mathrm{P}\left\{ Y_{1}(t_{1})\leq u_{1}\right\} ,
%&& G_{t_{1},s_{1}}^{(1)}(u_{1},v_{1})=\mathrm{P}\left\{ Y_{1}(t_{1})\leq
%u_{1},Y_{1}(s_{1})\leq v_{1}\right\} ,\\
&
G_{t_{2}}^{(2)}(u_{2})=\mathrm{P}\left\{ Y_{2}(t_{2})\leq u_{2}\right\} .%,
%&& G_{t_{2},s_{2}}^{(2)}(u_{2},v_{2})=%
%\mathrm{P}\left\{ Y_{2}(t_{2})\leq u_{2},Y_{2}(s_{2})\leq v_{2}\right\} .
\end{align*}

We know that for a PRF 
\begin{equation*}
\mathrm{E}\Delta _{s_{1},s_{2}}N(t_{1},t_{2})=\mathrm{E}N(1,1) \left(
t_{1}-s_{1}\right) \left( t_{2}-s_{2}\right) =\mathrm{Var}\Delta
_{s_{1},s_{2}}N(t_{1},t_{2});
\end{equation*}
\begin{equation*}
\mathrm{E}\left( \Delta _{s_{1},s_{2}}N(t_{1},t_{2})\right) ^{2}=\mathrm{E}N(1,1)
\left( t_{1}-s_{1}\right) \left( t_{2}-s_{2}\right) +\left[ \mathrm{E}N(1,1)\left(
t_{1}-s_{1}\right) \left( t_{2}-s_{2}\right) \right] ^{2}.
\end{equation*}

To prove 1) we use simple conditioning arguments: 
\begin{equation*}
\mathrm{E}Z(t_{1},t_{2})=\int_{0}^{\infty }\int_{0}^{\infty }u\ v\ \mathrm{E}%
N(1,1)G_{t_{1}}^{(1)}(du)G_{t_{2}}^{(2)}(dv)=\mathrm{E}N(1,1)\mathrm{E}%
Y_{1}(t_{1})\mathrm{E}Y_{2}(t_{2}).
\end{equation*}

Let us prove 2).

For the variance, we have 
\begin{eqnarray*}
\mathrm{Var}Z(t_{1},t_{2}) &=&\mathrm{E}\left(
N(Y_{1}(t_{1}),Y_{2}(t_{2})\right) ^{2}-\left( \mathrm{E}%
N(Y_{1}(t_{1}),Y_{2}(t_{2})\right) ^{2} \\
&=&\int_{0}^{\infty }\int_{0}^{\infty }\left( (\mathrm{E}N(u_{1},u_{2}))^{2}+%
\mathrm{Var}N(u_{1},u_{2})\right)
G_{t_{1}}^{(1)}(du_{1})G_{t_{2}}^{(2)}(du_{2}) \\
&&-\left( \mathrm{E}N(1,1)\mathrm{E}Y_{1}(t_{1})\mathrm{E}%
Y_{2}(t_{2})\right) ^{2} \\
&=&\int_{0}^{\infty }\int_{0}^{\infty }\left[ \left( \mathrm{E}N(1,1)\right)
^{2}u_{1}^{2}u_{2}^{2}+\mathrm{Var}N(1,1)u_{1}u_{2}\right]
G_{t_{1}}^{(1)}(du_{1})G_{t_{2}}^{(2)}(du_{2}) \\
&&-\left( \mathrm{E}N(1,1)\mathrm{E}Y_{1}(t_{1})\mathrm{E}%
Y_{2}(t_{2})\right) ^{2} \\
&=&\left( \mathrm{E}N(1,1)\right) ^{2}\mathrm{E}Y_{1}^{2}(t_{1})\mathrm{E}%
Y_{2}^{2}(t_{2})+\mathrm{Var}N(1,1)\mathrm{E}Y_{1}(t_{1})\mathrm{E}%
Y_{2}(t_{2}) \\
&&-\left( \mathrm{E}N(1,1)\mathrm{E}Y_{1}(t_{1})\mathrm{E}%
Y_{2}(t_{2})\right) ^{2} \\
&=&\left( \mathrm{E}N(1,1)\right) ^{2}\left\{ \mathrm{E}Y_{1}^{2}(t_{1})%
\mathrm{E}Y_{2}^{2}(t_{2})-(\mathrm{E}Y_{1}(t_{1}))^{2}(\mathrm{E}%
Y_{2}(t_{2}))^{2}\right\} \\
&&+\mathrm{Var}N(1,1)\mathrm{E}Y_{1}(t_{1})\mathrm{E}Y_{2}(t_{2}).
\end{eqnarray*}
To compute the covariance structure, first we consider the case when $%
s_{1}<t_{1},\ s_{2}<t_{2}$.
Then 
\begin{align*}
&\mathrm{E}N(s_{1},s_{2})N(t_{1},t_{2}) \\
&=\mathrm{E} \Big(N(s_{1},s_{2}) %
\Big\{N(t_{1},t_{2})-N(t_{1},s_{2})-N(s_{1},t_{2})+N(s_{1},s_{2})\\
& \qquad \qquad +N(t_{1},s_{2})+N(s_{1},t_{2})-N(s_{1},s_{2}) \Big\}\Big) \\
&=\mathrm{E}\Delta _{s_{1},s_{2}}N(t_{1},t_{2})\mathrm{E}N(s_{1},s_{2})+%
\mathrm{E}N(t_{1},s_{2})N(s_{1},s_{2})+\mathrm{E}%
N(s_{1},t_{2})N(s_{1},s_{2})-\mathrm{E}N^{2}(s_{1},s_{2}).
\end{align*}
Using the facts that 
\begin{align*}
&\mathrm{E}\Delta _{s_{1},s_{2}}N(t_{1},t_{2})\mathrm{E}%
N(s_{1},s_{2}) =(t_{1}-s_{1})(t_{2}-s_{2})\left[ \mathrm{E}N(1,1)\right]
^{2}s_{1}s_{2},
\\
&
\begin{aligned}
\mathrm{E}N(t_{1},s_{2})N(s_{1},s_{2}) &=\mathrm{E}%
\{N(t_{1},s_{2})-N(s_{1},s_{2})+N(s_{1},s_{2})\}N(s_{1},s_{2})
\\
&=\mathrm{E}\Delta _{s_{1},0}N(t_{1},s_{2})\mathrm{E}N(s_{1},s_{2})+\mathrm{E}%
N^{2}(s_{1},s_{2})\\
&=\left[ \mathrm{E}N(1,1)\right]
^{2}(t_{1}-s_{1})s_{1}s_{2}^{2}+\mathrm{E}N^{2}(s_{1},s_{2}),
\end{aligned}%
\intertext{it is easy to obtain}%
&\mathrm{E}N(s_{1},s_{2})N(t_{1},t_{2})=\left[ \mathrm{E}N(1,1)\right]
^{2}t_{1}t_{2}s_{1}s_{2}+s_{1}s_{2}\mathrm{Var}N(1,1).
\end{align*}
Since the processes $N,Y_{1},Y_{2}$ are independent, a conditioning argument
yields (\ref{3.5}) and (\ref{3.5b}). In a similar way, one can consider the case $s_{1}>t_{1}, s_{2}<t_{2}$.
\end{proof}

\begin{proof}[Proof of Proposition~\ref{prop:4.2}]
 It follows from Theorem~\ref{thm:A1} and Proposition~\ref{prop:2.1}.
\end{proof}

%\bibliography{article_biblio}
%\bibliographystyle{imsart-number}

\end{document}